\documentclass[12pt,a4paper]{article}
\usepackage{mathrsfs}
\usepackage{epsfig, graphicx}
\usepackage{latexsym,amsfonts,amsbsy,amssymb}
\usepackage{amsmath,amsthm}
\usepackage{color}
\usepackage[colorlinks, citecolor=blue]{hyperref}

\usepackage{color}

\makeatletter % '@' is now a normal "letter" for TeX

\@addtoreset{equation}{section}
\makeatother % '@' is restored as a "non-letter" character for TeX
\allowdisplaybreaks % For long formula

\newtheorem{theorem}{Theorem}[section]
\newtheorem{lemma}{Lemma}[section]

\newtheorem{remark}{Remark}[section]
\newtheorem{algorithm}{Algorithm}[section]

\newtheorem{proposition}{Proposition}[section]

\begin{document}
\title{A Multigrid Method Based On Shifted-Inverse Power Technique
for Eigenvalue  Problems\footnote{This work is
supported in part by the National Science Foundation of China
(NSFC 91330202, 11001259, 11371026, 11201501, 11301437, 11031006,11171251,11201501
2011CB309703),  the Natural Science Foundation of Fujian Province
of China(No.2013J05015), the National Basic Research Program (2012CB955804),
Tianjin University of Finance and Economics (ZD1302), the National Center for Mathematics and
Interdisciplinary Science, CAS and the President Foundation of AMSS-CAS.}}
\author{
Hongtao Chen\footnote{School of Mathematical Sciences, Xiamen University,
 Xiamen 361005, China(chenht@xmu.edu.cn)}, \ \
Yunhui He\footnote{LSEC, ICMSEC, Academy of Mathematics and Systems
Science, Chinese Academy of Sciences,  Beijing 100190,
China(heyunhui@lsec.cc.ac.cn)},\ \ Yu Li\footnote{Research Center
for Mathematics and Economics, Tianjin University of Finance
and Economics, Tianjin 300222, China(liyu@lsec.cc.ac.cn)}\ \ and\ \
Hehu Xie\footnote{LSEC, NCMIS, Institute
of Computational Mathematics, Academy of Mathematics and Systems
Science, Chinese Academy of Sciences, Beijing 100190,
China(hhxie@lsec.cc.ac.cn)} }
\date{}
\maketitle
\begin{abstract}
A multigrid method is proposed in this paper to solve
eigenvalue problems by the finite element method based on the
shifted-inverse power iteration technique. With this scheme, solving
eigenvalue problem is transformed to a series of nonsingular solutions
of boundary value problems on multilevel meshes. Since replacing the difficult
eigenvalue solving by the easier solution of boundary value problems,
the multigrid way can improve the overall efficiency of the eigenvalue problem solving.
Some numerical experiments are presented to validate the efficiency of this new method.


\end{abstract}

\section{Introduction}
Solving large scale eigenvalue problems becomes a fundamental problem
in modern science and engineering society. However, it is always a very
difficult task to solve high-dimensional eigenvalue problems which come from
physical and chemical sciences. About the solution of eigenvalue problems,
\cite{BrandtMcCormickRuge, Hackbusch,Hackbusch_Book,HuCheng,Shaidurov,YangBi}
and the references cited therein give some types of multigrid schemes which
couple the multigrid method with the Rayleigh quotient iteration technique.
The involved almost singular linear problems in these methods lead to the
numerical unstability. So it is required to design some special solver for
these almost singular linear problems \cite{Hackbusch_Book,Shaidurov}.

The aim of this paper is to present a type of shifted-inverse power iteration
method to solve the eigenvalue problem based on the multigrid technique.
Compared with the existed
works \cite{BrandtMcCormickRuge, Hackbusch,Hackbusch_Book,Shaidurov,YangBi},
the method here is a generalization of the Rayleigh quotient iteration
technique and does not need to solve the singular linear problems.
Recently, we propose another type of multigrid method \cite{Xie_IMA,Xie_JCP}
based on the multilevel correction method \cite{LinXie} which transforms the solution
of eigenvalue problem to a series of boundary value problem solving and
eigenvalue problem solving in a very coarse space. The proposed method here
does not need to solve the eigenvalue problem in the coarse space but needs the
eigenvalue problem possessing good eigenvalue separations and the initial approximation
 having good accuracy.
The standard Galerkin finite element method for eigenvalue problems
has been extensively investigated, e.g. Babu\v{s}ka and Osborn
\cite{Babuska2,BabuskaOsborn}, Chatelin \cite{Chatelin} and
references cited therein. Here we adopt some basic results in these
papers for our analysis.
The corresponding error and computational work discussion of the proposed iteration
scheme will be analyzed. Based
on the analysis, the new method can obtain optimal errors with an optimal
computational work when we can solve the associated linear problems with the
optimal complexity.

In order to describe our method clearly,
we give the following simple Laplace eigenvalue problem to
illustrate the main idea in this paper (sections 3 and 4).

Find $(\lambda,u)\in \mathcal{R}\times H_0^1(\Omega)$ such that
\begin{equation}\label{problem}
\left\{
\begin{array}{rcl}
-\Delta u&=&\lambda u,\quad{\rm in}\ \Omega,\\
u&=&0,\quad\ \  {\rm on}\ \partial\Omega,\\
\int_{\Omega} u^2d\Omega&=&1,
\end{array}
\right.
\end{equation}
where $\Omega\subset\mathcal{R}^2$ is a bounded domain with
Lipschitz boundary $\partial\Omega$ and $\Delta$ denotes the Laplace
operator.

First, we construct a series of finite element spaces $V_{h_1}$, $ V_{h_2}$,
$\cdots$, $V_{h_n}$ which are subspaces of $H_0^1(\Omega)$ and defined on
the corresponding series of multilevel
meshes $\mathcal{T}_{h_k}\ (k=1,2,\cdots n)$ such that $V_{h_1}\subset
V_{h_2}\subset \cdots \subset V_{h_n}$ and $h_k=h_{k-1}/\beta$
(see, e.g., \cite{BrennerScott,Ciarlet}). Our multigrid algorithm
to obtain the approximation of the eigenpairs can be defined as
follows (see sections 3 and 4):
\begin{enumerate}
\item Solve an eigenvalue problem in the coarsest space $V_{h_1}$:

Find $(\lambda_{h_1}, u_{h_1})\in \mathcal{R}\times V_{h_1}$ such that
$\|\nabla u_{h_1}\|_0=1$ and
\begin{eqnarray*}
\int_{\Omega}\nabla u_{h_1}\nabla v_{h_1}d\Omega &=&
\lambda_{h_1}\int_{\Omega}u_{h_1}v_{h_1}d\Omega,\ \ \ \forall v_{h_1}\in V_{h_1}.
\end{eqnarray*}
\item Do $k=1,\cdots, n-1$
\begin{itemize}
\item
Solve the following auxiliary boundary value  problem:

Find $\widetilde{u}_{h_{k+1}}\in V_{h_{k+1}}$ such that for any
$v_{h_{k+1}}\in V_{h_{k+1}}$
\begin{eqnarray*}
\int_{\Omega}\big(\nabla \widetilde{u}_{h_{k+1}}\nabla v_{h_{k+1}}
-\alpha_{k+1}\widetilde{u}_{h_{k+1}}v_{h_{k+1}}\big)d\Omega
&=&\int_{\Omega}u_{h_k}v_{h_{k+1}}d\Omega.
\end{eqnarray*}
\item
Do the normalization
\begin{eqnarray*}
u_{h_{k+1}} &=& \frac{\widetilde{u}_{h_{k+1}}}{\|\nabla\widetilde{u}_{h_{k+1}}\|_0}
\end{eqnarray*}
and compute the Rayleigh quotient for $u_{h_{k+1}}$
\begin{eqnarray*}
 \lambda_{h_{k+1}}&=&\frac{\|\nabla u_{h_{k+1}}\|_0^2}{\|u_{h_{k+1}}\|_0^2}.
\end{eqnarray*}
\end{itemize}
End Do
\end{enumerate}
If, for example, $\lambda_{h_1}$ is the approximation for the first eigenvalue of the
problem (\ref{problem}) at the first step and $\Omega$ is a convex domain, then we can establish the
following results by taking a suitable choice of $\alpha_{k+1}$ (see sections 3 and 4 for details)
\begin{eqnarray*}
\|\nabla (u-u_{h_n})\|_0
=\mathcal{O}(h_n), \ \ \ {\rm and}\ \ \
|\lambda-\lambda_{h_n}| =\mathcal{O}(h_n^2).
\end{eqnarray*}
%with the computational work $\mathcal{O}(N_n)$.
These two estimates mean that we obtain asymptotically optimal
errors.

In this method, we replace solving eigenvalue problem on the finest
finite element spaces by solving a series of boundary value problems
in the corresponding series of finite element spaces and an
eigenvalue problem in the initial finite element space.

An outline of the paper goes as follows. In Section 2, we introduce the
finite element method for the eigenvalue problem and give the corresponding basic
error estimates. A type of one shifted-inverse power iteration step is given in Section
3. In Section 4, we propose a type of multigrid
algorithm for solving the eigenvalue problem based on the shifted-inverse power iteration step.
The computational work estimate of the eigenvalue multigrid method is given in Section 5.
In Section 6, two numerical examples are presented to validate our
theoretical analysis. Some concluding remarks are given in the last section.

\section{Discretization by finite element method}
In this section, we introduce some notation and error estimates of
the finite element approximation for the eigenvalue problem.
The letter $C$ (with or without subscripts) denotes a generic
positive constant which may be different at its different occurrences through the paper.
For convenience, the symbols $\lesssim$, $\gtrsim$ and $\approx$
will be used in this paper. That $x_1\lesssim y_1, x_2\gtrsim y_2$
and $x_3\approx y_3$, mean that $x_1\leq C_1y_1$, $x_2 \geq c_2y_2$
and $c_3x_3\leq y_3\leq C_3x_3$ for some constants $C_1, c_2, c_3$
and $C_3$ that are independent of mesh sizes (see, e.g., \cite{Xu}).

Let $(V, \|\cdot\|)$ be a real Hilbert space with inner product
$(\cdot, \cdot)$ and norm $\|\cdot\|$, respectively. Let
$a(\cdot,\cdot)$, $b(\cdot,\cdot)$ be two symmetric bilinear forms
on $V \times V$ satisfying
\begin{eqnarray}
a(w, v) &\lesssim&\|w\|\|v\|, \ \ \forall w,v\in V, \label{Bounded}\\
\|w\|^2&\lesssim& a(w,w),\ \ \forall w\in V, \label{Coercive}\\
0&<&b(w,w), \ \ \forall 0\neq w\in V.
\end{eqnarray}
From (\ref{Bounded}) and (\ref{Coercive}), we know that
$\|\cdot\|_a:=a(\cdot,\cdot)^{1/2}$ and $\|\cdot\|$ are two
equivalent norms on $V$. We assume that the norm $\|\cdot\|$ is
relatively compact with respect to the norm
$\|\cdot\|_b:=b(\cdot,\cdot)^{1/2}$ in the sense that
any sequence which is bounded in $\|\cdot\|$, one can extract
a subsequence which is Cauchy with respect to $\|\cdot\|_b$.
We shall use $a(\cdot,\cdot)$
and $\|\cdot\|_a$, respectively, as the inner product and norm on $V$ in the rest
of this paper.

We assume that
$V_h\subset V$ is a family of finite-dimensional spaces that satisfy
the following assumption:

For any $w \in V$
\begin{eqnarray}\label{Approximation_Property}
\lim_{h\rightarrow0}\inf_{v_h\in V_h}\|w-v_h\|_a = 0.
\end{eqnarray}
Let $P_h$ be the finite element projection operator of $V$ onto
$V_h$ defined by
\begin{eqnarray}\label{Projection_Problem}
a(w - P_hw, v_h) = 0,\ \ \ \forall w\in V,\ \  \forall v_h\in
V_h.
\end{eqnarray}
Obviously
\begin{eqnarray}\label{Projection_Optimal}
\|P_hw\|_a\leq \|w\|_a,\ \ \ \forall w\in V.
\end{eqnarray}
For any $w\in V$, by (\ref{Approximation_Property}) we have
\begin{eqnarray}
\|w-P_hw\|_a&=&o(1),\ \ \ {\rm as}\ h\rightarrow 0.
\end{eqnarray}
Define $\eta_a(h)$ as
\begin{eqnarray}
\eta_a(h)=\sup_{f\in V',\|f\|_b=1}\inf_{v_h\in V_h}\|T f-v_h\|_a,
\end{eqnarray}
where $V'$ denotes the dual space of $V$ and the operator $T: V'\rightarrow V$
is defined as
\begin{eqnarray}
a(Tf,v)&=&b(f,v),\ \ \ \forall f\in V', \ \ \forall v\in V.
\end{eqnarray}
In order to derive the error estimate of eigenpair approximation in
the weak norm $\|\cdot\|_{b}$, we need the following weak norm
error estimate of the finite element projection operator $P_h$.
\begin{lemma}\label{Negative_norm_estimate_Lemma}
(\cite[Lemma 3.3 and Lemma 3.4]{BabuskaOsborn})
\begin{eqnarray}
\eta_a(h)\rightarrow0,\ \ \ \ {\rm as}\ h\rightarrow 0,
\end{eqnarray}
and
\begin{eqnarray}\label{Negative_norm_Error}
\|w-P_hw\|_{b}&\lesssim&\eta_a(h)\|w-P_hw\|_a,\ \ \ \forall w\in V.
\end{eqnarray}
\end{lemma}

In our methodology description, we are concerned with the following
general eigenvalue problem:

Find $(\lambda, u )\in \mathcal{R}\times V$ such that $a(u,u)=1$ and
\begin{eqnarray}
a(u,v)&=&\lambda b(u,v),\quad \forall v\in V. \label{weak_problem}
\end{eqnarray}

For the eigenvalue $\lambda$, there exists the following Rayleigh
quotient expression (see, e.g., \cite{Babuska2,BabuskaOsborn,XuZhou})
\begin{eqnarray}\label{Rayleigh_quotient}
\lambda=\frac{a(u,u)}{b(u,u)}.
\end{eqnarray}
From \cite{BabuskaOsborn,Chatelin}, we know the eigenvalue problem
(\ref{weak_problem}) has an eigenvalue sequence $\{\lambda_j \}:$
$$0<\lambda_1\leq \lambda_2\leq\cdots\leq\lambda_k\leq\cdots,\ \ \
\lim_{k\rightarrow\infty}\lambda_k=\infty,$$ and the associated
eigenfunctions
$$u_1,u_2,\cdots,u_k,\cdots,$$
where $b(u_i,u_j)=\delta_{ij}$. In the sequence $\{\lambda_j\}$, the
$\{\lambda_j\}$ are repeated according to their geometric multiplicity.

Let $M(\lambda_i)$ denote the eigenfunction space corresponding to the
eigenvalue $\lambda_i$ which is defined by
\begin{eqnarray}
M(\lambda_i)&=&\big\{w\in V: w\ {\rm is\ an\ eigenfunction\ of\
(\ref{weak_problem})
\ corresponding} \nonumber\\
&&\ \ \ \ {\rm to}\ \lambda_i\big\}. %\  {\rm and}\ \|w\|_a=1\big\}.
\end{eqnarray}
From \cite{Babuska2,BabuskaOsborn}, each eigenvalue $\lambda_i$ can
be defined as follows
\begin{eqnarray*}
\lambda_i&=&\inf_{v\in V\atop v\perp M(\lambda_j)\ {\rm for}\
\lambda_j<\lambda_i} \frac{a(v,v)}{b(v,v)}.
\end{eqnarray*}

Now, let us define the finite element approximations of the problem
(\ref{weak_problem}). First we generate a shape-regular
decomposition of the computing domain $\Omega\subset \mathcal{R}^d\
(d=2,3)$ into triangles or rectangles for $d=2$ (tetrahedrons or
hexahedrons for $d=3$). The diameter of a cell $K\in\mathcal{T}_h$
is denoted by $h_K$. The mesh diameter $h$ describes the maximum
diameter of all cells $K\in\mathcal{T}_h$. Based on the mesh
$\mathcal{T}_h$, we can construct a finite element space denoted by
$V_h\subset V$.

Then we define the approximation for the eigenpair $({\lambda},u)$ of
(\ref{weak_problem}) by the finite element method as:

Find $(\bar{\lambda}_h, \bar{u}_h)\in \mathcal{R}\times V_h$ such that
 $a(\bar{u}_h,\bar{u}_h)=1$ and
\begin{eqnarray}\label{weak_problem_Discrete}
a(\bar{u}_h,v_h)&=&\bar{\lambda}_hb(\bar{u}_h,v_h),\quad\ \  \ \forall v_h\in V_h.
\end{eqnarray}

From (\ref{weak_problem_Discrete}), we know the following
Rayleigh quotient expression for $\bar{\lambda}_h$ holds
(see, e.g., \cite{Babuska2,BabuskaOsborn,XuZhou})
\begin{eqnarray}\label{eigenvalue_Rayleigh}
\bar{\lambda}_h &=&\frac{a(\bar{u}_h,\bar{u}_h)}{b(\bar{u}_h,\bar{u}_h)}.
\end{eqnarray}
Similarly, we know from \cite{BabuskaOsborn,Chatelin} the eigenvalue
problem (\ref{weak_problem_Discrete}) has eigenvalues
$$0<\bar{\lambda}_{1,h}\leq \bar{\lambda}_{2,h}\leq\cdots\leq \bar{\lambda}_{k,h}\leq\cdots\leq \bar{\lambda}_{N_h,h},$$
and the corresponding eigenfunctions
$$\bar{u}_{1,h}, \bar{u}_{2,h},\cdots, \bar{u}_{k,h}, \cdots, \bar{u}_{N_h,h},$$
where $b(\bar{u}_{i,h},\bar{u}_{j,h})=\delta_{ij}, 1\leq i,j\leq N_h$ ($N_h$ is
the dimension of the finite element space $V_h$).

From the minimum-maximum principle (see, e.g., \cite{Babuska2,BabuskaOsborn}),
the following upper bound result holds
$${\lambda}_i\leq \bar{\lambda}_{i,h}, \ \ \ i=1,2,\cdots, N_h.$$
Similarly, let $M_h(\lambda_i)$ denote the approximate eigenfunction space  corresponding to the
eigenvalue $\lambda_i$ which is defined by
\begin{eqnarray}
M_h(\lambda_i)&=&\big\{w_h\in V_h: w_h\ {\rm is\ an\ eigenfunction\ of\
(\ref{weak_problem_Discrete})} \nonumber\\
&&\ \ \ \ {\rm corresponding\ to}\ \lambda_i\big\}. %\  {\rm and}\ \|w_h\|_a=1\big\}.
\end{eqnarray}
%}
From \cite{Babuska2,BabuskaOsborn}, each eigenvalue $\lambda_{i,h}$ can be defined as follows
\begin{eqnarray}\label{Eigenvalue_h_definition}
\lambda_{i,h}&=&\inf_{v_h\in V_h\atop v_h\perp M_h(\lambda_j)\ {\rm for}\ \lambda_j<\lambda_i} \frac{a(v_h,v_h)}{b(v_h,v_h)}.
\end{eqnarray}

Then we define
\begin{eqnarray}
\delta_h(\lambda_i)=\sup_{w\in M(\lambda_i),\|w\|_a=1}\inf_{v_h\in
V_h}\|w-v_h\|_a.
\end{eqnarray}
For the analysis in this paper, we introduce the spectral projection as follows \cite{BabuskaOsborn}:
\begin{eqnarray}\label{Spectral_Projection_Operator}
\Pi_{1,h}&=&\frac{1}{2\pi {\rm i}}\int_{\Gamma}\big(z-P_hT\big)^{-1}dz,
\end{eqnarray}
where $\Gamma$ is a Jordan curve in $\mathcal{C}$ that encloses the eigenvalue $\lambda_1$
and no other eigenvalues. % surrounds

For the eigenpair approximations by finite element method, there
exist the following error estimates.
\begin{proposition}(\cite[Lemma 3.7, (3.29b)]{Babuska2}, \cite[P. 699]{BabuskaOsborn} and
\cite{Chatelin})\label{Error_estimate_Proposition}

\noindent(i) For any eigenfunction approximation $\bar u_{i,h}$ of
(\ref{weak_problem_Discrete}) $(i = 1, 2, \cdots, N_h)$, there is an
eigenfunction $u_i$ of (\ref{weak_problem}) corresponding to
$\lambda_i$ such that $\|u_i\|_a = 1$ and
\begin{eqnarray}\label{Eigenfunction_Error}
\|u_i-\bar{u}_{i,h}\|_a&\leq& C\delta_h(\lambda_i).
\end{eqnarray}
Furthermore,
\begin{eqnarray}\label{Eigenfunction_Error_Nagative}
\|u_i- \bar{u}_{i,h}\|_{b} \leq C\eta_a(h)\|u_i - \bar u_{i,h}\|_a.
\end{eqnarray}
(ii) For each eigenvalue, we have
\begin{eqnarray}
\lambda_i \leq \bar{\lambda}_{i,h}\leq \lambda_i +
C\delta_h^2(\lambda_i).
\end{eqnarray}
Here and hereafter $C$ is some constant depending on $\lambda_i$ but independent of  the mesh size $h$.
\end{proposition}

\section{One shifted-inverse power iteration step with multigrid method}
In this section, we present a type of one shifted-inverse power iteration step to improve the
accuracy of the given eigenvalue and eigenfunction approximations.
This iteration method only contains solving auxiliary boundary value problems
in the finer finite element space.

To analyze our method, we introduce the error expansion of the
eigenvalue by the Rayleigh quotient formula which comes from
\cite{Babuska2,BabuskaOsborn,LinYan,XuZhou}.
\begin{lemma}(\cite{Babuska2,BabuskaOsborn,LinYan,XuZhou})\label{Rayleigh_Quotient_error_theorem}
Assume $(\lambda,u)$ is a true solution of the eigenvalue problem
(\ref{weak_problem}) and  $0\neq \psi\in V$. Let us define
\begin{eqnarray}\label{rayleighw}
\widehat{\lambda}=\frac{a(\psi,\psi)}{b(\psi,\psi)}.
\end{eqnarray}
Then we have
\begin{eqnarray}\label{rayexpan}
\widehat{\lambda}-\lambda
&=&\frac{a(u-\psi,u-\psi)}{b(\psi,\psi)}-\lambda
\frac{b(u-\psi,u-\psi)}{b(\psi,\psi)}.
\end{eqnarray}
\end{lemma}
For simplicity, here we only state the numerical method for the first
eigenvalue $\lambda_1$. Assume we have obtained an eigenpair approximation
$(\lambda_{1,h_k},u_{1,h_k})\in\mathcal{R}\times V_{h_k}$ with $\|u_{1,h_k}\|_a=1$.
Now we introduce a type of iteration step to improve the accuracy of the
current eigenpair approximation $(\lambda_{1,h_k},u_{1,h_k})$. Let
$V_{h_{k+1}}\subset V$ be a finer finite element space such that
$V_{h_k}\subset V_{h_{k+1}}$. Based on this finer finite element space,
we define the following one shifted-inverse power iteration step.

\begin{algorithm}\label{Correction_Step}
One Shifted-inverse Power Iteration Step

\begin{enumerate}
\item Solve the following boundary value problem:

Find $\widehat{u}_{1,h_{k+1}}\in V_{h_{k+1}}$ such that for any $v_{h_{k+1}}\in V_{h_{k+1}}$
\begin{eqnarray}\label{aux_problem}
a(\widehat{u}_{1,h_{k+1}},v_{h_{k+1}})-\alpha_{1,k+1}b(\widehat{u}_{1,h_{k+1}},v_{h_{k+1}})
=b(u_{1,h_k},v_{h_{k+1}}).
\end{eqnarray}
\item Do the normalization for $\widehat{u}_{1,h_{k+1}}$ as
\begin{eqnarray}\label{Normalization}
u_{1,h_{k+1}}=\frac{\widehat{u}_{1,h_{k+1}}}{\|\widehat{u}_{1,h_{k+1}}\|_a}
\end{eqnarray}
and compute the Rayleigh quotient for $u_{1,h_{k+1}}$
\begin{eqnarray}\label{Rayleigh_Quotient_u_h_k_1}
\lambda_{1,h_{k+1}}&=&\frac{a(u_{1,h_{k+1}},u_{1,h_{k+1}})}{b(u_{1,h_{k+1}},u_{1,h_{k+1}})}.
\end{eqnarray}
\end{enumerate}
Then we obtain a new eigenpair approximation $(\lambda_{1,h_{k+1}},u_{1,h_{k+1}})\in \mathcal{R}\times V_{h_{k+1}}$.
Summarize the above two steps into
\begin{eqnarray*}
(\lambda_{1,h_{k+1}},u_{1,h_{k+1}})={\it
Correction}(\alpha_{1,k+1},u_{1,h_k},V_{h_{k+1}}).
\end{eqnarray*}
\end{algorithm}
%--------------------------------------------------------------------------------------
%\begin{remark}\label{Multiple_Remark}
%If the concerned eigenvalue $\lambda$ is multiple, we choose the smallest eigenvalue approximation
%of $\lambda$ as $\lambda_{h_{k+1}}$ and its associated eigenfunction approximation as $u_{h_{k+1}}$
%which has the biggest component in the direction of $\widetilde{u}_{h_{k+1}}$.
%\end{remark}
%--------------------------------------------------------------------------------------------------
\begin{theorem}\label{Error_Estimate_One_Correction_Theorem}
When $\alpha_{1,k+1}\neq \bar{\lambda}_{1,h_{k+1}}$ and $\alpha_{1,k+1} < \bar{\lambda}_{2,h_{k+1}}$,
after one correction step, the resultant approximation
$u_{1,h_{k+1}}\in V_{h_{k+1}}$
has the following error estimates
\begin{eqnarray}\label{Error_k_k+1_1}
\|u_{1,h_{k+1}}-\Pi_{1,h_{k+1}}u_{1,h_{k+1}}\|_a
&\leq& \frac{\theta_{1,k+1}\|u_{1,h_k}-\Pi_{1,h_{k+1}}u_{1,h_k}\|_a}{1-(1+\theta_{1,k+1})\|u_{1,h_k}-\Pi_{1,h_{k+1}}u_{1,h_k}\|_a},
%\frac{|\bar{\lambda}_{1,h_{k+1}}-\alpha_{1,k+1}|}{\bar{\lambda}_{2,h_{k+1}}-\alpha_{1,k+1}}
%\|u_{1,h_k}-\bar{u}_{1,h_{k+1}}\|_a,
\end{eqnarray}
where
\begin{eqnarray}\label{Definition_Theta}
\theta_{1,k+1}&=&\frac{|\bar{\lambda}_{1,h_{k+1}}-\alpha_{1,k+1}|}{\bar{\lambda}_{2,h_{k+1}}-\alpha_{1,k+1}}.
\end{eqnarray}
\end{theorem}
\begin{proof}
Let us define $\gamma=\bar{\lambda}_{1,h_{k+1}}-\alpha_{1,k+1}$, $\widetilde{u}_{1,h_{k+1}}=\gamma \widehat{u}_{1,h_{k+1}}$
and $w_{h_{k+1}}:=\widetilde{u}_{1,h_{k+1}}-\Pi_{1,h_{k+1}}u_{1,h_k}$.
Then we know
\begin{eqnarray}\label{aux_problem_2}
a(\widetilde{u}_{1,h_{k+1}},v_{h_{k+1}})-\alpha_{1,k+1}b(\widetilde{u}_{1,h_{k+1}},v_{h_{k+1}})
=\gamma b(u_{1,h_{k}},v_{h_{k+1}}),\ \forall v_{h_{k+1}}\in V_{h_{k+1}}.
\end{eqnarray}
From $u_{1,h_k}-\Pi_{1,h_{k+1}}u_{1,h_k}\perp M_{h_{k+1}}(\lambda_1)$, $\Pi_{1,h_{k+1}}u_{1,h_k}\in M_{h_{k+1}}(\lambda_1)$,
(\ref{aux_problem_2}) and for any $v_{h_{k+1}}\in V_{h_{k+1}}$
\begin{eqnarray}\label{Eigenvalue_Problem_Direct}
a(\Pi_{1,h_{k+1}}u_{1,h_k},v_{h_{k+1}})-\alpha_{1,k+1}b(\Pi_{1,h_{k+1}}u_{1,h_k},v_{h_{k+1}})
=\gamma b(\Pi_{1,h_{k+1}}u_{1,h_k},v_{h_{k+1}}),
\end{eqnarray}
we have $w_{h_{k+1}}\perp M_{h_{k+1}}(\lambda_1)$ and $a(w_{h_{k+1}}, \Pi_{1,h_{k+1}}u_{1,h_k})=0$.

From (\ref{Eigenvalue_h_definition}), (\ref{aux_problem_2})
and (\ref{Eigenvalue_Problem_Direct}),
the following estimates hold
\begin{eqnarray*}
\|w_{h_{k+1}}\|_a^2&=
& \alpha_{1,k+1}b(w_{h_{k+1}},w_{h_{k+1}})
+ \gamma b(u_{1,h_k}-\Pi_{1,h_{k+1}}u_{1,h_k},w_{h_{k+1}})\nonumber\\
&\leq&\frac{\alpha_{1,k+1}}{\bar{\lambda}_{2,h_{k+1}}}\|w_{h_{k+1}}\|_a^2
+|\gamma| \|u_{1,h_k}-\Pi_{1,h_{k+1}}u_{1,h_k}\|_b\|w_{h_{k+1}}\|_b\nonumber\\
&\leq&\frac{\alpha_{1,k+1}}{\bar{\lambda}_{2,h_{k+1}}}\|w_{h_{k+1}}\|_a^2
+\frac{|\gamma|}{\bar{\lambda}_{2,h_{k+1}}}\|u_{1,h_k}-\Pi_{1,h_{k+1}}u_{1,h_k}\|_a\|w_{h_{k+1}}\|_a.
\end{eqnarray*}
Then we have
\begin{eqnarray}\label{Estimate_u_tilde_u_h_k+1}
\|w_{h_{k+1}}\|_a &\leq
&\frac{\bar{|\lambda}_{1,h_{k+1}}-\alpha_{1,k+1}|}{\bar{\lambda}_{2,h_{k+1}}-\alpha_{1,k+1}}
\|u_{1,h_k}-\Pi_{1,h_{k+1}}u_{1,h_k}\|_a.
\end{eqnarray}
From (\ref{Estimate_u_tilde_u_h_k+1}), $\|u_{1,h_k}\|_a=1$ and
 $\widetilde{u}_{1,h_{k+1}}=u_{1,h_k}+\Pi_{1,h_{k+1}}u_{1,h_k}-u_{1,h_k}+\widetilde{u}_{1,h_{k+1}}-\Pi_{1,h_{k+1}}u_{1,h_k}$,
  the following estimates hold
\begin{eqnarray}
\|\widetilde{u}_{1,h_{k+1}}\|_a&\geq& \|u_{1,h_k}\|_a-\|\Pi_{1,h_{k+1}}u_{1,h_k}-u_{1,h_k}\|_a-\|\widetilde{u}_{1,h_{k+1}}-\Pi_{1,h_{k+1}}u_{1,h_k}\|_a\nonumber\\
&\geq& 1-(1+\theta_{1,k+1})\|u_{1,h_k}-\Pi_{1,h_{k+1}}u_{1,h_k}\|_a.
\end{eqnarray}
Then we have
\begin{eqnarray}
\frac{1}{\|\widetilde{u}_{1,h_{k+1}}\|_a}&\leq& \frac{1}{1-(1+\theta_{1,k+1})\|u_{1,h_k}-\Pi_{1,h_{k+1}}u_{1,h_k}\|_a}.
\end{eqnarray}
Since $u_{1,h_{k+1}}=\widetilde{u}_{1,h_{k+1}}/\|\widetilde{u}_{1,h_{k+1}}\|_a$ and
$w_{h_{k+1}}\perp M_{h_{k+1}}(\lambda_1)$, the following estimates hold
\begin{eqnarray}
\|u_{1,h_{k+1}}-\Pi_{1,h_{k+1}}u_{1,h_{k+1}}\|_a&=&
\frac{\|\widetilde{u}_{1,h_{k+1}}-\Pi_{1,h_{k+1}}u_{1,h_k}\|_a}{\|\widetilde{u}_{1,h_{k+1}}\|_a}\nonumber\\
&\leq& \frac{\theta_{1,k+1}\|u_{1,h_k}-\Pi_{1,h_{k+1}}u_{1,h_k}\|_a}{1-(1+\theta_{1,k+1})\|u_{1,h_k}-\Pi_{1,h_{k+1}}u_{1,h_k}\|_a},
\end{eqnarray}
which means we have obtained the desired result (\ref{Error_k_k+1_1}) and the proof is complete..
%Similarly, we can
%chose $\bar{u}_{1,h_k}$ such that $u_{1,h_k}-\bar{u}_{1,h_k}\perp M_{h_k}(\lambda_1)$.
%Then from (\ref{Eigenvalue_h_definition}), (\ref{aux_problem_2})
%and (\ref{Eigenvalue_Problem_Direct}),  the following estimates hold
%\begin{eqnarray*}
%\|w_{h_{k+1}}\|_a^2
%&=& \alpha_{1,k+1}b(w_{h_{k+1}},w_{h_{k+1}})
%+|\gamma| b(u_{1,h_k}-\Pi_{1,h_{k+1}}u_{1,h_k},w_{h_{k+1}})\nonumber\\
%&\leq&\frac{\alpha_{1,k+1}}{\bar{\lambda}_{2,h_{k+1}}}\|w_{h_{k+1}}\|_a^2
%+|\gamma| b(u_{1,h_k}-\bar{u}_{1,h_k},w_{h_{k+1}})\nonumber\\
%&&+|\gamma| b(\bar{u}_{1,h_k}-\Pi_{1,h_{k+1}}u_{1,h_k},w_{h_{k+1}})\nonumber\\
%&\leq&\frac{\alpha_{1,k+1}}{\bar{\lambda}_{2,h_{k+1}}}\|w_{h_{k+1}}\|_a^2
%+\frac{|\gamma|}{\sqrt{\bar{\lambda}_{2,h_{k+1}}\bar{\lambda}_{2,h_{k}}}}
%\|u_{1,h_k}-\bar{u}_{1,h_k}\|_a\|w_{h_{k+1}}\|_a\nonumber\\
%&&+\frac{|\gamma|}{\sqrt{\bar{\lambda}_{2,h_{k+1}}}}\|\bar{u}_{1,h_k}-\Pi_{1,h_{k+1}}u_{1,h_k}\|_b\|w_{h_{k+1}}\|_a.
%\end{eqnarray*}
%It means that
%\begin{eqnarray}\label{Estimate_u_tilde_u_h_k+1}
%\|w_{h_{k+1}}\|_a
%&\leq&\frac{|\bar{\lambda}_{1,h_{k+1}}-\alpha_{1,k+1}|}{\bar{\lambda}_{2,h_{k+1}}-\alpha_{1,k+1}}
%\left(
%\sqrt{\frac{\bar{\lambda}_{2,h_{k+1}}}{\bar{\lambda}_{2,h_{k}}}}\|u_{1,h_k}-\bar{u}_{1,h_k}\|_a\right.\nonumber\\
%&&\ \ \ \ \ \ \ \ \ \ \left.+\sqrt{\bar{\lambda}_{2,h_{k+1}}}\|\bar{u}_{1,h_k}-\Pi_{1,h_{k+1}}u_{1,h_k}\|_b
%\right).
%\end{eqnarray}
%This is the desired result (\ref{Error_k_k+1}) and the proof is complete.
\end{proof}
\begin{remark}
Let us discuss two choices of $\alpha_{1,k+1}$.
\begin{enumerate}
\item If $\alpha_{1,{k+1}}=0$, i.e. we do not use shift, actually Algorithm \ref{Correction_Step}
is just to repeat the well-known two grid method \cite{XuZhou}. From (\ref{Error_k_k+1_1}), we have
\begin{equation}
\|u_{1,{h_{k+1}}}-\Pi_{1,h_{k+1}}u_{1,{h_{k+1}}} \|_a\leq \kappa_{1,k+1}\|u_{1,{h_{k}}}-\Pi_{1,h_{k+1}}u_{1,h_k}\|_a,
\end{equation}
where
\begin{eqnarray*}
\kappa_{1,k+1}&=&\frac{\frac{\bar\lambda_{1,h_{k+1}}}{\bar\lambda_{2,h_{k+1}}}}
{1-\Big(1+\frac{\bar\lambda_{1,h_{k+1}}}{\bar\lambda_{2,h_{k+1}}}\Big)\|u_{1,{h_{k}}}-\Pi_{1,h_{k+1}}u_{1,h_k}\|_a}.
\end{eqnarray*}
This indicates that the two grid step has linear convergence speed rate when
$\|u_{1,{h_{k}}}-\Pi_{1,h_{k+1}}u_{1,h_k}\|_a$ is small enough.
But if the gap between $\bar\lambda_{1,h_{k+1}}$ and $\bar\lambda_{2,h_{k+1}}$ is small, the convergence is slow.

\item %then according to
From Lemma \ref{Rayleigh_Quotient_error_theorem}, the following estimate holds
\begin{eqnarray*}
&&\lambda_{1,h_{k}}-\bar\lambda_{1,h_{k+1}}\\
&=&\frac{a(\Pi_{1,h_{k+1}}u_{1,h_k}-u_{1,h_{k}},\Pi_{1,h_{k+1}}u_{1,h_k}-u_{1,h_{k}})}{b(u_{1,h_{k}},u_{1,h_{k}})}\nonumber\\
&&\ \ \ \ -\bar\lambda_{1,h_{k+1}}\frac{b(\Pi_{1,h_{k+1}}u_{1,h_k}-u_{1,h_{k}},\Pi_{1,h_{k+1}}u_{1,h_k}-u_{1,h_{k}})}{b(u_{1,h_{k}},u_{1,h_{k}})}\\
&\lesssim &\|u_{1,h_{k}}-\Pi_{1,h_{k+1}}u_{1,h_k}\|^2_a.
\end{eqnarray*}
If $\alpha_{1,h_{k+1}}=\lambda_{1,h_{k}}$, from (\ref{Error_k_k+1_1}), we have
\begin{equation}
\|u_{1,h_{k+1}}-\Pi_{1,h_{k+1}} u_{1,h_{k+1}}\|_a\lesssim\|u_{1,h_{k}}-\Pi_{1,h_{k+1}}u_{1,h_k}\|^3_a,
\end{equation}
which means that the Rayleigh quotient iteration has cubic convergence rate.
\end{enumerate}
\end{remark}
%\begin{remark}\label{Choice_Alpha}

The suitable choice for $\alpha_{1,k+1}$ sometimes is not
so easy to obtain,  since it depends on $\bar{\lambda}_{1,h_{k+1}}$ and $\bar{\lambda}_{2,h_{k+1}}$ which are unknown.
But from (\ref{Error_k_k+1_1}), if
\begin{eqnarray*}
0< \frac{|\bar{\lambda}_{1,h_{k+1}}-\alpha_{1,k+1}|}{\bar{\lambda}_{2,h_{k+1}}-\alpha_{1,k+1}}<1,
\end{eqnarray*}
the accuracy for the solution of (\ref{aux_problem}) can be improved through more times iteration.
Then we can design the following modified one multi shifted-inverse power iteration step.

\begin{algorithm}\label{Correction_Step_Multi_Steps}
Multi Shifted-inverse Power Iteration Step
\begin{enumerate}
\item Set $u_{1,h_{k+1}}^0 =u_{1,h_k}$.

\item Do $j=0, \cdots, \ell-1$

\begin{itemize}
\item Solve the following boundary value problem:

Find $\widehat{u}_{1,h_{k+1}}^{j+1}\in V_{h_{k+1}}$ such that for any $v_{h_{k+1}}\in V_{h_{k+1}}$
\begin{eqnarray}\label{aux_problem_modified}
a(\widehat{u}_{1,h_{k+1}}^{j+1},v_{h_{k+1}})-\alpha_{1,k+1}^{j+1}b(\widehat{u}_{1,h_{k+1}}^{j+1},v_{h_{k+1}})
=b(u_{1,h_{k+1}}^j,v_{h_{k+1}}).
\end{eqnarray}
\item Normalize $\widehat{u}_{1,h_{k+1}}^{j+1}$ by
$$u_{1,h_{k+1}}^{j+1}=\frac{\widehat{u}_{1,h_{k+1}}^{j+1}}{\|\widehat{u}_{1,h_{k+1}}^{j+1}\|_a}$$
and compute the Rayleigh quotient for $u_{1,h_{k+1}}^{j+1}$
\begin{equation}\label{Rayleigh_Quotient_u_h_k_1}
\lambda_{1,h_{k+1}}^{j+1}=\frac{a(u_{1,h_{k+1}}^{j+1},u_{1,h_{k+1}}^{j+1})}
{b(u_{1,h_{k+1}}^{j+1},u_{1,h_{k+1}}^{j+1})}.
\end{equation}
\end{itemize}
End Do
\item Set $u_{1,h_{k+1}}=u_{1,h_{k+1}}^{\ell}$ and $\lambda_{1,h_{k+1}}=\lambda_{1,h_{k+1}}^{\ell}$.
\end{enumerate}
Then we obtain a new eigenpair approximation
$(\lambda_{1,h_{k+1}},u_{1,h_{k+1}})\in \mathcal{R}\times V_{h_{k+1}}$.
Summarize the above two steps into
\begin{eqnarray*}
(\lambda_{1,h_{k+1}},u_{1,h_{k+1}})={\it
Correction}(\{\alpha_{1,k+1}^j\}_{j=1}^{\ell},u_{1,h_k},V_{h_{k+1}}).
\end{eqnarray*}
\end{algorithm}
In Algorithm \ref{Correction_Step_Multi_Steps},
we can adjust $\ell$ such that the following estimate satisfies
\begin{eqnarray}\label{Error_k_k+1_Muli_Steps}
\|u_{1,h_{k+1}}-\Pi_{1,h_{k+1}}{u}_{1,h_{k+1}}\|_a
< \frac{1}{\beta} \|u_{1,h_k}-\Pi_{1,h_{k+1}}{u}_{1,h_k}\|_a,
\end{eqnarray}
where $\beta$ is a constant defined in the following section.
In fact, $\ell$ can be very small ($\ell=2$ or $3$) and the modified
iteration step makes the choice of $\alpha_{1,k+1}$ become not so sharp as
in Algorithm \ref{Correction_Step} and it improves the stability of the iteration step.

\section{Multigrid scheme for the eigenvalue problem}
In this section, we introduce a type of multigrid
scheme based on the {\it One Shifted-inverse Power Iteration Step} defined in Algorithm
\ref{Correction_Step} or \ref{Correction_Step_Multi_Steps}. This type of multigrid
method can obtain the optimal error estimate as same as solving the eigenvalue problem
 directly in the finest finite element space.

In order to do multigrid scheme, we define a sequence of triangulations $\mathcal{T}_{h_k}$
of $\Omega$ determined as follows. Suppose $\mathcal{T}_{h_1}$ is given and let
$\mathcal{T}_{h_k}$ be obtained from $\mathcal{T}_{h_{k-1}}$ via regular refinement
 (produce $\beta^d$  subelements) such that
$$h_k=\frac{1}{\beta}h_{k-1}.$$
Based on this sequence of meshes, we construct the corresponding linear finite element spaces such that
\begin{eqnarray}\label{FEM_Space_Series}
V_{h_1}\subset V_{h_2}\subset\cdots\subset V_{h_n},
\end{eqnarray}
and the following relation of approximation errors hold
\begin{eqnarray}\label{Error_k_k_1}
\eta_a(h_k)\approx\frac{1}{\beta}\eta_a(h_{k-1}),\ \ \ \
\delta_{h_k}(\lambda)\approx\frac{1}{\beta}\delta_{h_{k-1}}(\lambda),\ \ \ k=2,\ldots,n.
\end{eqnarray}
From the spectral projection definition by (\ref{Spectral_Projection_Operator}), we have
\begin{eqnarray}\label{Spectral_Projection_Estimate}
\|\Pi_{1,h_k}u_{1,h_k}-\Pi_{1,h_{k+1}}u_{1,h_k}\|_a&\leq& C_4\delta_{h_k}(\lambda_1), \ \ k=1,\ldots,n-1,
\end{eqnarray}
where the constant $C_4$ is independent of the mesh size $h_k$.

\begin{algorithm}\label{Multi_Correction}
Eigenvalue Multigrid Scheme
\begin{enumerate}
\item Construct a series of nested finite element
spaces $V_{h_1}, V_{h_2},\cdots,V_{h_n}$ such that
(\ref{FEM_Space_Series}) and (\ref{Error_k_k_1}) hold.
\item Solve the following eigenvalue problem:

Find $(\lambda_{1,h_1},u_{1,h_1})\in \mathcal{R}\times V_{h_1}$ such that
$a(u_{1,h_1},u_{1,h_1})=1$ and
\begin{eqnarray}\label{Initial_Eigen_Problem}
a(u_{1,h_1},v_{h_1})&=&\lambda_{1,h_1}b(u_{1,h_1},v_{h_1}),\ \ \ \ \forall v_{h_1}\in V_{h_1}.
\end{eqnarray}

\item  Do $k=1,\cdots,n-1$

%\begin{itemize}
%\item[]
 Obtain a new eigenpair approximation
$(\lambda_{1,h_{k+1}},u_{1,h_{k+1}})\in \mathcal{R}\times V_{h_{k+1}}$
by a correction step
\begin{eqnarray}
(\lambda_{1,h_{k+1}},u_{1,h_{k+1}})=Correction(\alpha_{1,k+1},u_{1,h_k},V_{h_{k+1}}).
\end{eqnarray}
%\end{itemize}
End Do
\end{enumerate}
Finally, we obtain an eigenpair approximation
$(\lambda_{1,h_n},u_{1,h_n})\in \mathcal{R}\times V_{h_n}$.
\end{algorithm}
%------------------------------------------------------------------------------------------------------
For simplicity in this paper, we assume the following estimates hold
\begin{eqnarray}\label{Error_Order}
\eta_a(h_k)\approx h_k,\ \ \ \ \delta_{h_k}(\lambda_1)\approx h_k,\ \ \ \ k=1,\ldots,n.
\end{eqnarray}
%--------------------------------------------------------------------------------------------------------
\begin{theorem}
After implementing Algorithm \ref{Multi_Correction},
%there exists an eigenpair $(\lambda_1,u_1) \in\mathcal{R}\times M(\lambda_1)$ of
%(\ref{weak_problem}) such that
 the resultant eigenpair approximation $(\lambda_{1,h_n},u_{1,h_n})$ has the following
error estimates
\begin{eqnarray}
\|u_{1,h_n}-\Pi_{1,h_n}u_{1,h_n}\|_a &\leq&C_5\delta_{h_n}(\lambda_1),\label{Multi_Correction_Err_fun}\\
|\lambda_{1,h_n}-\bar{\lambda}_{1,h_n}|&\leq&C_5\delta_{h_n}^2(\lambda_1),\label{Multi_Correction_Err_eigen}
\end{eqnarray}
when the mesh size $h_1$ is small enough and $\alpha_{1,{k+1}}$ is chosen as follows
\begin{eqnarray}\label{alpha_defintion_1}
\alpha_{1,k+1}=\max\left\{0,\frac{2(1+C_4/C_5)\beta\lambda_{1,h_k}-\lambda_{2,h_1}}{2(1+C_4/C_5)\beta-1}\right\}, \ \ \ 
k=1, \ldots, n-1,
\end{eqnarray}
where $C_5$ is a constant not less than the constant $C$ in (\ref{Eigenfunction_Error}).

Then there exists a constant $C_6$ such that the following final convergence results hold
\begin{eqnarray}
\|u_1-u_{1,h_n}\|_a &\leq&C_6\delta_{h_n}(\lambda_1),\label{Multi_Correction_Err_fun_Final}\\
|\lambda_1-\lambda_{1,h_n}|&\leq&C_6\delta_{h_n}^2(\lambda_1).\label{Multi_Correction_Err_eigen_Final}
\end{eqnarray}

%Furthermore, if we choose $\alpha_{1,k+1}$ as follows
%\begin{eqnarray}\label{alpha_defintion_2}
%\alpha_{1,k+1}=\max\left\{0,\frac{3\beta^2\lambda_{1,h_k}-\lambda_{2,h_1}}{3\beta^2-1}\right\}, \ \ \
%k=1, \cdots, n-1.
%\end{eqnarray}
%We have
%\begin{eqnarray}
%\|u_{1,h_n}-\bar{u}_{1,h_n}\|_a &\leq&
%C\eta_a(h_n)\delta_{h_n}(\lambda_1),\label{Multi_Correction_Err_fun_Superapp}\\
%|\lambda_{1,h_n}-\bar{\lambda}_{1,h_n}|&\leq&
%C\eta_a^2(h)\delta_{h_n}^2(\lambda_1).\label{Multi_Correction_Err_eigen_Superapp}
%\end{eqnarray}
\end{theorem}
\begin{proof}
First, it is obvious that $2(1+C_4/C_5)\beta > 1$.
If we choose the $\alpha_{1,k+1}$ as in (\ref{alpha_defintion_1}) and
$2(1+C_4/C_5)\beta\lambda_{1,h_k}-\lambda_{2,h_1}>0$,
then $\alpha_{1,k+1}=\frac{2(1+C_4/C_5)\beta\lambda_{1,h_k}
-\lambda_{2,h_1}}{2(1+C_4/C_5)\beta-1}$ and the following estimates hold
\begin{eqnarray}\label{Rate_Convergence_1}
\theta_{1,k+1}&=&\frac{|\bar{\lambda}_{1,h_{k+1}}-\alpha_{1,k+1}|}
{\bar{\lambda}_{2,h_{k+1}}-\alpha_{1,k+1}}\nonumber\\
&=&\frac{|\lambda_{2,h_1}-\bar{\lambda}_{1,h_{k+1}}+2(1+C_4/C_5)\beta(\bar{\lambda}_{1,h_{k+1}}-\lambda_{1,h_k})|}
{2(1+C_4/C_5)\beta(\bar{\lambda}_{2,h_{k+1}}-\lambda_{1,h_k})+\lambda_{2,h_1}-\bar{\lambda}_{2,h_{k+1}}}\nonumber\\
&=&\frac{1}{2(1+C_4/C_5)\beta}+\mathcal{O}(h_1^2).
\end{eqnarray}
%when $h_1$ is small enough.

If we choose $\alpha_{1,k+1}$ as in (\ref{alpha_defintion_1}) and $2(1+C_4/C_5)\beta\lambda_{1,h_k}-\lambda_{2,h_1}<0$,
 then $\alpha_{1,k+1}=0$ and we have
 \begin{eqnarray}\label{Rate_Convergence_2}
\theta_{1,k+1}=\frac{|\bar{\lambda}_{1,h_{k+1}}-\alpha_{1,k+1}|}{\bar{\lambda}_{2,h_{k+1}}-\alpha_{1,k+1}}
&=&\frac{|\lambda_{1,h_k}+(\bar{\lambda}_{1,h_{k+1}}-\lambda_{1,h_k})|}
{\lambda_{2,h_1} +\bar{\lambda}_{2,h_{k+1}} -\lambda_{2,h_1}}\nonumber\\
&=&\frac{1}{2(1+C_4/C_5)\beta}+\mathcal{O}(h_1^2). %<\frac{1}{\beta},
\end{eqnarray}
%when $h_1$ is small enough.

Now, let us prove (\ref{Multi_Correction_Err_fun}) by the method of induction.
First, it is obvious that (\ref{Multi_Correction_Err_fun}) holds for $n=1$.
Then we assume that (\ref{Multi_Correction_Err_fun}) holds for $n=k$. It means we have the following estimate
\begin{eqnarray}\label{Assumption}
\|u_{1,h_k}-\Pi_{1,h_k}u_{1,h_k}\|_a&\leq& C_5\delta_{h_k}(\lambda_1).
\end{eqnarray}
Now let us consider the case of $n=k+1$.
Combining (\ref{Spectral_Projection_Estimate}), (\ref{Assumption})
and the triangle inequality leads to the following estimates
\begin{eqnarray}\label{Estimate_1}
\|u_{1,h_k}-\Pi_{1,h_{k+1}}u_{1,h_k}\|_a &\leq& \|u_{1,h_k}-\Pi_{1,h_k}u_{1,h_k}\|_a +
\|\Pi_{1,h_k}u_{1,h_k}-\Pi_{1,h_{k+1}}u_{1,h_k}\|_a\nonumber\\
&\leq&  \|u_{1,h_k}-\Pi_{1,h_k}u_{1,h_k}\|_a +C_4\delta_{h_k}(\lambda_1)\nonumber\\
&\leq& C_5\Big(1+\frac{C_4}{C_5}\Big)\delta_{h_k}(\lambda_1).
\end{eqnarray}
From (\ref{Definition_Theta}), (\ref{Rate_Convergence_1})-(\ref{Rate_Convergence_2})
and (\ref{Estimate_1}), we have
\begin{eqnarray}\label{Estimate_2}
\frac{\theta_{1,k+1}}{1-(1+\theta_{1,k+1})\|u_{1,h_k}-\Pi_{1,h_{k+1}}u_{1,h_k}\|_a}
&=&\frac{1}{2(1+C_4/C_5)\beta}+\mathcal{O}(h_1)\nonumber\\
&<&\frac{1}{(1+C_4/C_5)\beta},
\end{eqnarray}
when $h_1$ is small enough.

 From Theorem \ref{Error_Estimate_One_Correction_Theorem}, (\ref{Estimate_1}) and
 (\ref{Estimate_2}), we have
\begin{eqnarray}\label{Error_k_k+1_1_New}
\|u_{1,h_{k+1}}-\Pi_{1,h_{k+1}}u_{1,h_{k+1}}\|_a &\leq& \frac{C_5}{\beta}\delta_{h_k}(\lambda_1)
=C_5\delta_{h_{k+1}}(\lambda_1).
\end{eqnarray}
This means that the result (\ref{Multi_Correction_Err_fun}) also holds for $n=k+1$.
Thus we prove the desired result (\ref{Multi_Correction_Err_fun}).
From Lemma \ref{Rayleigh_Quotient_error_theorem}
and (\ref{Multi_Correction_Err_fun}), we can obtain the desired result
 (\ref{Multi_Correction_Err_eigen}). Finally, (\ref{Multi_Correction_Err_fun_Final}) and
 (\ref{Multi_Correction_Err_eigen_Final})
  can be proved from (\ref{Multi_Correction_Err_fun}), (\ref{Multi_Correction_Err_eigen})
 and the triangle inequality.
%where
%\begin{eqnarray}\label{Definition_Theta}
%\theta_{1,k+1}&=&\frac{|\bar{\lambda}_{1,h_{k+1}}-\alpha_{1,k+1}|}{\bar{\lambda}_{2,h_{k+1}}-\alpha_{1,k+1}}.
%\end{eqnarray}
%
%consider the value of $\theta_{1,k+1}$ for $k=1,\cdots,n-1$.
%
%From Lemma \ref{Negative_norm_estimate_Lemma}, (\ref{Rate_Convergence_1})-(\ref{Rate_Convergence_2}),
%Theorem \ref{Error_Estimate_One_Correction_Theorem} and
%the recursive relation (\ref{Error_k_k_1}),
%the final eigenfunction approximation $u_{1,h_n}$
%%and the direct finite element approximation $\bar{u}_{1,h_n}$
%has following estimates
%\begin{eqnarray}\label{Error_u_h_n_Multi_Correction}
%\|u_{1,h_n}-\Pi_{1,h_n}u_{1,h_n}\|_a&\leq&\Pi_{k=2}^n\theta_{1,k}\|u_{1,h_1}-\Pi_{1,h_1}u_{1,h_1}\|_a
%%&\leq&\frac{1}{\beta}\|u_{1,h_{n-1}}-\bar{u}_{1,h_{n-1}}\|_a
%%+C\eta_a(h_{n-1})\delta_{h_{n-1}}(\lambda_1)\nonumber\\
%%&\leq&C \sum_{k=1}^{n-1}\Big(\frac{1}{\beta}\Big)^{n-1-k}\eta_a(h_{k})\delta_{h_k}(\lambda_1)\nonumber\\
%%&\leq&C \left( \sum_{k=1}^{n-1}\Big(\frac{1}{\beta}\Big)^{n-1-k}\beta^{n-k}\Big(\frac{1}{\beta}\Big)^{k-1}\right)
%%\eta_a(h_1)\delta_{h_n}(\lambda_1)\nonumber\\
%%&\leq& C\frac{\beta^2}{\beta-1}\eta_a(h_1)\delta_{h_n}(\lambda_1).
%\end{eqnarray}
%Then the desired result (\ref{Multi_Correction_Err_fun}) can be derived by combining
%(\ref{Eigenfunction_Error}), (\ref{Error_u_h_n_Multi_Correction})
%and the triangle inequality. From Lemma \ref{Rayleigh_Quotient_error_theorem}
% and (\ref{Multi_Correction_Err_fun}), we can obtain the desired result
% (\ref{Multi_Correction_Err_eigen}).
\end{proof}

\begin{remark}
We also investigate the condition of the boundary value problem with different
choices of $\alpha_{1,k+1}$. If we choose $\alpha_{1,k+1}$ as (\ref{alpha_defintion_1}) and
$2(1+C_4/C_5)\beta\lambda_{1,h_k}-\lambda_{2,h_1}>0$, then $\alpha_{1,k+1}<\bar{\lambda}_{1,h_{k+1}}$
when the mesh size $h_1$ is small enough. It means that (\ref{aux_problem})
is a symmetric positive definite linear equation and its condition has the following estimates
\begin{eqnarray}
\frac{\bar{\lambda}_{n,h_{k+1}}-\alpha_{1,k+1}}{\bar{\lambda}_{1,h_{k+1}}-\alpha_{1,k+1}} \approx
\frac{{\lambda}_{1}}{{\lambda}_{1}-\alpha_{1,k+1}}
\frac{\bar{\lambda}_{n,h_{k+1}}}{\bar{\lambda}_{1,{h_{k+1}}}}
\approx \frac{2(1+C_4/C_5)\beta\lambda_1}{\lambda_2-\lambda_1}{\rm cond}(A),
\end{eqnarray}
where $A$ and $\bar{\lambda}_{n,h_{k+1}}$ denote the stiff matrix and
the largest eigenvalue approximation of (\ref{weak_problem_Discrete}), respectively,
in the finite element space $V_{h_{k+1}}$, the convergence results of $\lambda_{1,h_k}$,
$\bar{\lambda}_{1,h_k}$ and $\lambda_{2,h_1}$ are used.
\end{remark}

\section{Work estimate of eigenvalue multigrid scheme}
In this section, we turn our attention to the estimate of computational work
for Algorithm \ref{Multi_Correction}. We will show that
Algorithm \ref{Multi_Correction} makes solving the eigenvalue problem need almost the
same work as solving the corresponding boundary value problem if we adopt the
multigrid method to solve the involved linear problems (\ref{aux_problem})
(see, e.g., \cite{BrennerScott,Hackbusch,Hackbusch_Book,Xu,Xu_Two_Grid}).

First, we define the dimension of each level linear
finite element space as $N_k:={\rm dim}V_{h_k}$. Then we have
\begin{eqnarray}\label{relation_dimension}
N_k \thickapprox\Big(\frac{1}{\beta}\Big)^{d(n-k)}N_n,
\ \ \ k=1,2,\ldots, n.
\end{eqnarray}
\begin{theorem}
Assume the eigenvalue problem solved in the coarse space $V_{h_1}$ needs work
$\mathcal{O}(M_{h_1})$ and the work of the multigrid solver
 in each level space $V_{h_k}$ is
$\mathcal{O}(N_k)$ for $k=2,3,\cdots,n$. Then the total work involved in
 Algorithm \ref{Multi_Correction}
 is $\mathcal{O}(N_n+M_{h_1})$. Furthermore, the complexity
will be $\mathcal{O}(N_n)$ provided $M_{h_1}\leq N_n$.
\end{theorem}

\begin{proof}
Let $W_k$ denote the work of the correction step in the $k$-th finite element
space $V_{h_k}$. Then with the correction definition, we have
\begin{eqnarray}\label{work_k}
W_k&=&\mathcal{O}(N_k).
\end{eqnarray}
Iterating (\ref{work_k}) and using the fact (\ref{relation_dimension}), we obtain
\begin{eqnarray}\label{Work_Estimate}
{\rm Total\ work} &=&\sum_{k=1}^nW_k= \mathcal{O}\Big(M_{h_1}+\sum_{k=2}^nN_k\Big)
=\mathcal{O}\Big(M_{h_1}+\sum_{k=2}^nN_k\Big)\nonumber\\
&=&\mathcal{O}\Big(M_{h_1}+\sum_{k=2}^n\Big(\frac{1}{\beta}\Big)^{d(n-k)}N_n\Big)
=\mathcal{O}(N_n+M_{h_1}).
\end{eqnarray}
This is the desired result $\mathcal{O}(N_n+M_{h_1})$ and the
one $\mathcal{O}(N_n)$  can be obtained by the condition $M_{h_1}\leq N_n$.
\end{proof}

\section{Numerical results}
In this section, two numerical examples are presented to illustrate the
efficiency of the multigrid scheme proposed in this
paper.

\subsection{Model eigenvalue problem}
Here we give the numerical results of the multigrid
scheme for the Laplace eigenvalue problem
on the two dimensional domain $\Omega=(0,1 )\times (0, 1)$.  The sequence of
finite element spaces is constructed by
using linear element on the series of meshes which are produced by
regular refinement with $\beta =2$ (producing $\beta^d$ congruent subelements).
In this example, we use two meshes which are generated by Delaunay method as the initial mesh
$\mathcal{T}_{h_1}$ to produce two sequences of finite element spaces for investigating
 the convergence behaviors.
Figure \ref{Initial_Mesh} shows the corresponding
initial meshes: one is coarse and the other is fine.

Algorithm \ref{Multi_Correction} is applied to solve the eigenvalue problem.
For comparison, we also solve the eigenvalue problem by the direct method.
\begin{figure}[htb]
\centering
\includegraphics[width=5cm,height=5cm]{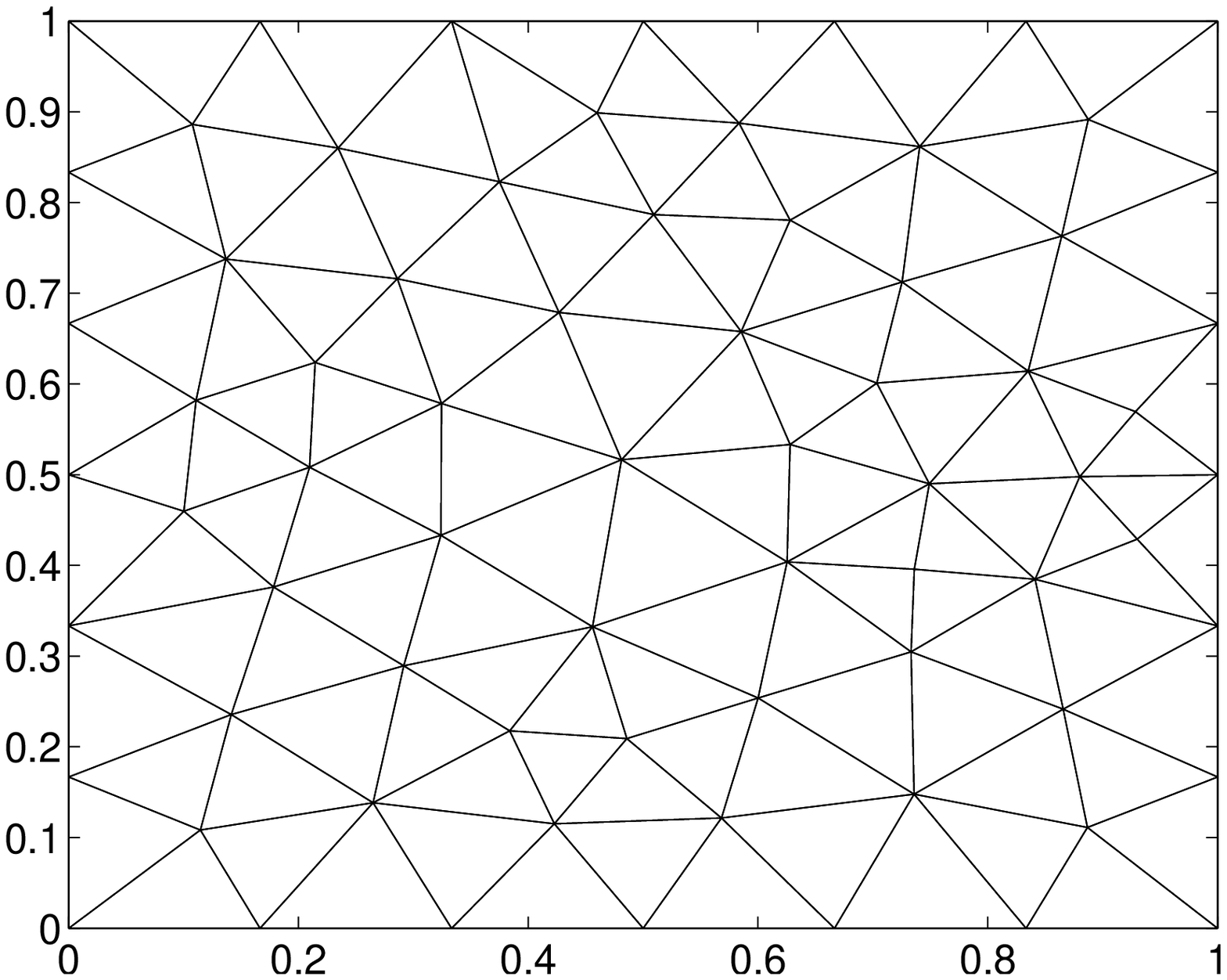}
\includegraphics[width=5cm,height=5cm]{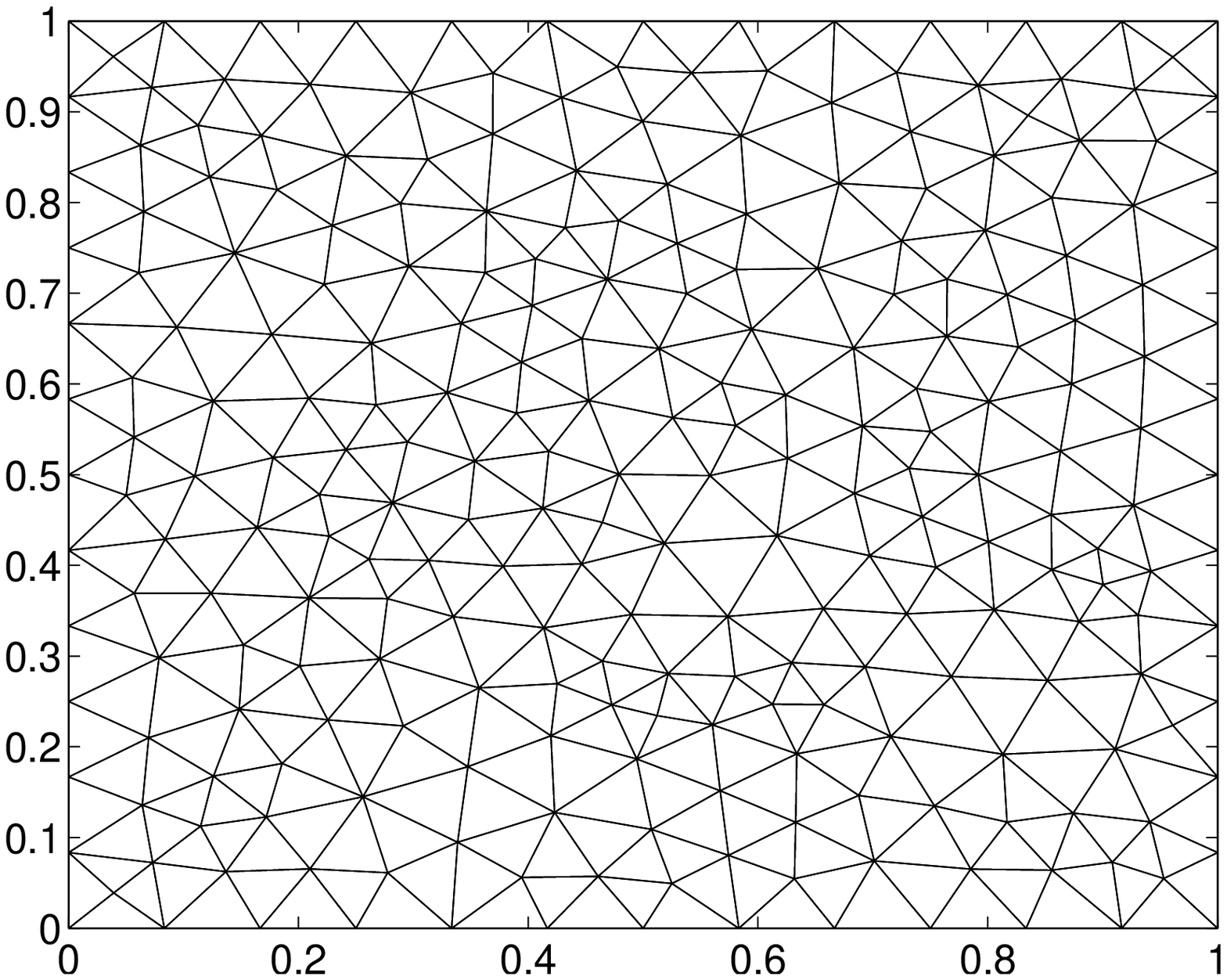}
\caption{\small\texttt The coarse and fine initial meshes for Example 1.}
\label{Initial_Mesh}
\end{figure}

Figure \ref{numerical_multi_grid_2D}
gives the corresponding numerical results for the first eigenvalue
$\lambda_1=2\pi^2$ and the corresponding eigenfunction on the two initial meshes
 illustrated in Figure \ref{Initial_Mesh}.
\begin{figure}[htb]
\centering
\includegraphics[width=7cm,height=6cm]{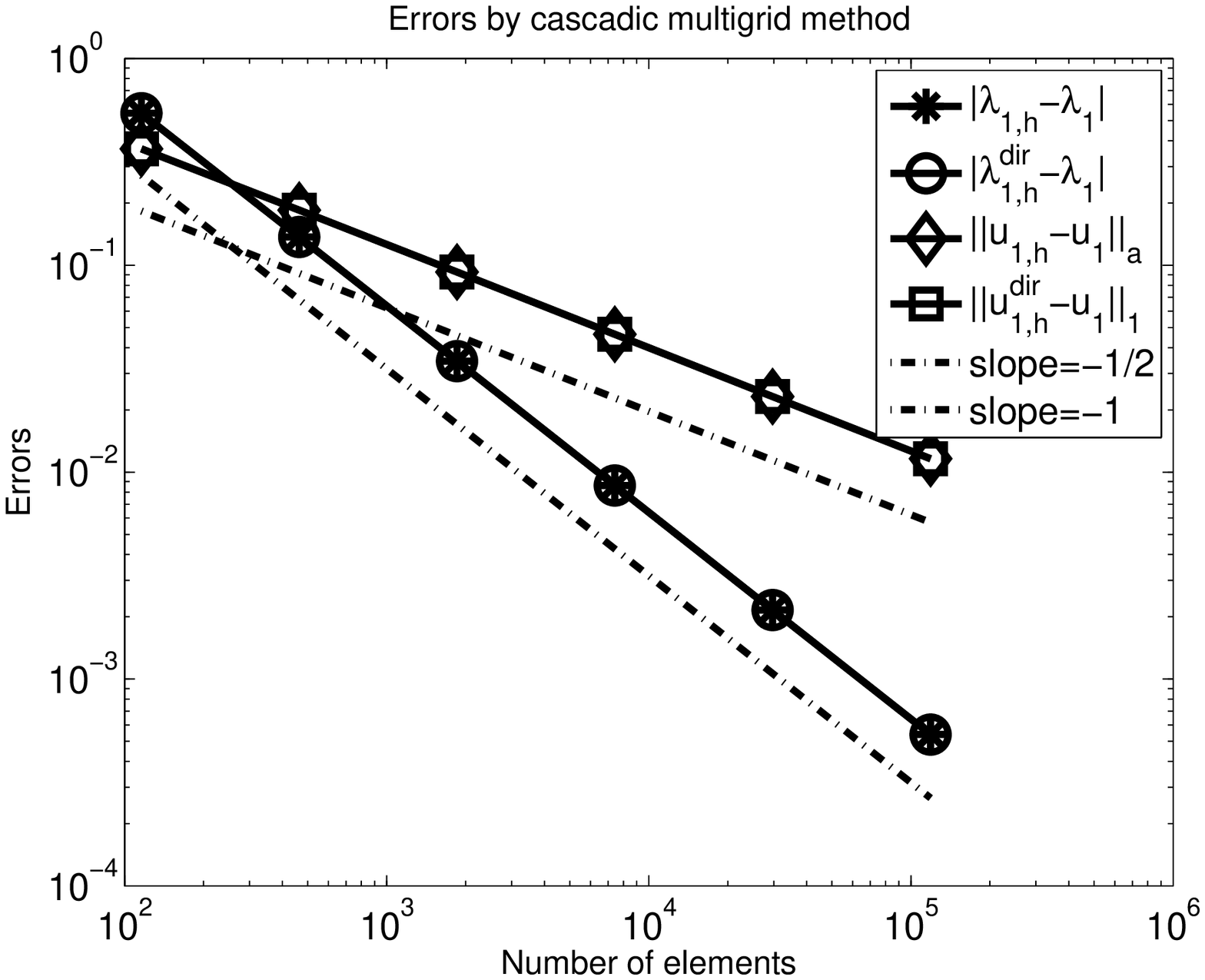}
\includegraphics[width=7cm,height=6cm]{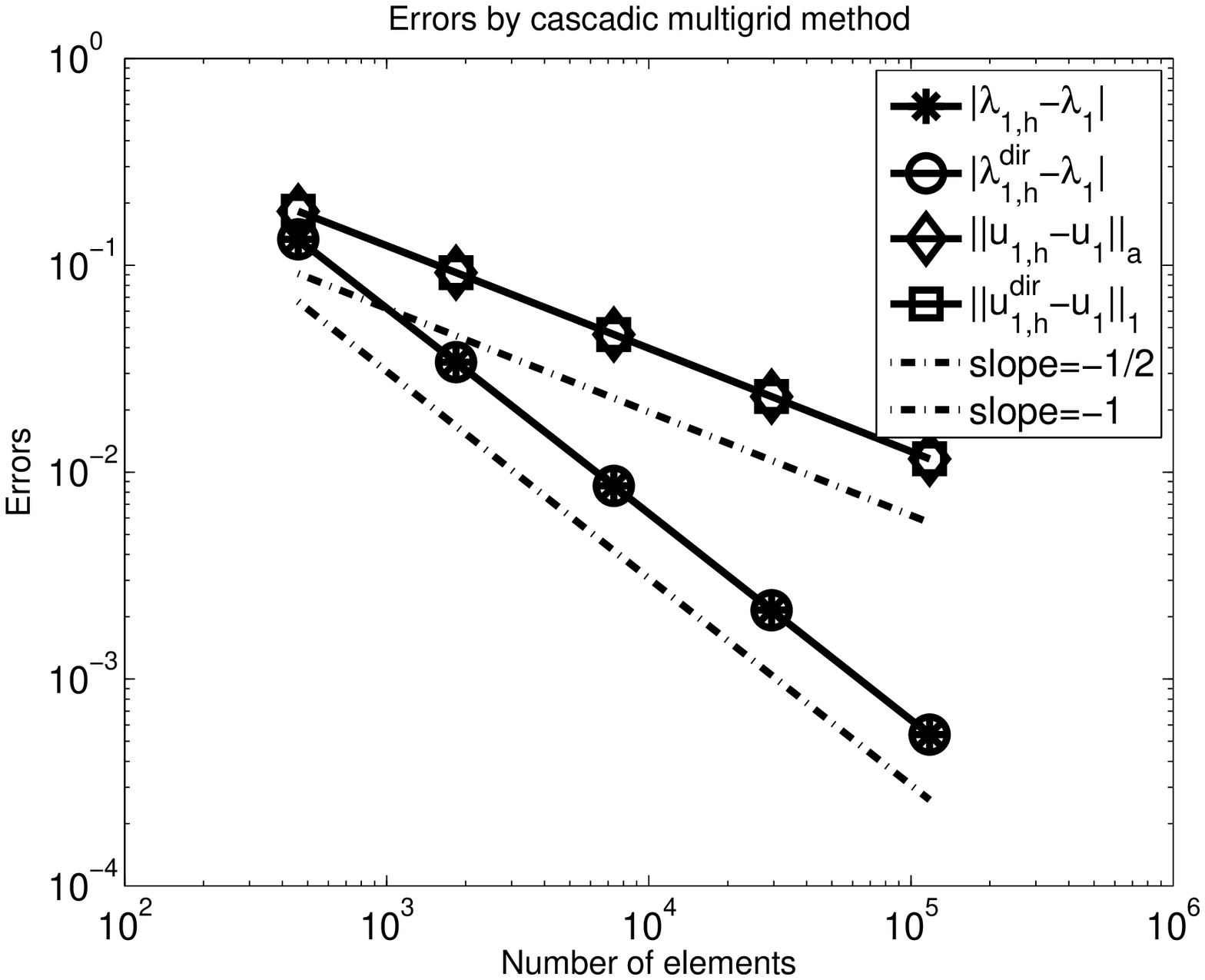}
\caption{\small\texttt The errors of the multigrid
algorithm for the first eigenvalue $2\pi^2$ and the corresponding eigenfunction,
where $u_h^{\rm dir}$ and $\lambda_h^{\rm dir}$ denote the eigenfunction
 and eigenvalue approximation by direct eigenvalue solving
 (The left subfigure is for the coarse initial mesh in the left of Figure \ref{Initial_Mesh}
and the right one for the fine initial mesh in the right of Figure \ref{Initial_Mesh}).}
\label{numerical_multi_grid_2D}
\end{figure}

From Figure \ref{numerical_multi_grid_2D},
we find the multigrid scheme can obtain
the optimal error estimates as same as the direct eigenvalue solving method for the eigenvalue and the
corresponding eigenfunction approximations.

In order to show the efficiency of the proposed method in this paper, we also compare
the running time of Algorithm \ref{Multi_Correction} and the direct eigenvalue solving. Here
we choose the package ARPACK \cite{LehoucqMaschhoffSorensenYang,LehoucqSorensenYang}
as the direct eigenvalue solving tool to do the comparison. Here the
conjugate gradient (CG) iteration and algebraic multigrid preconditioner are
adopted to act as the linear solver in the ARPACK. In Algorithm \ref{Multi_Correction},
 we only choose the CG iteration as the linear solver.
Both methods are running in the same machine {\it PowerEdge R720}
 with the Linux system. The corresponding results are shown in Table \ref{Table_CPU_Time_Exam_1}
 which implies that Algorithm \ref{Multi_Correction} can improve the efficiency of
 eigenvalue problem solving.
\begin{table}[http]
\centering
\caption{The CPU time comparison between Algorithm \ref{Multi_Correction} and
direct eigenvalue solver. Here the number of elements in $V_{h_1}$ is $3968$.}
\label{Table_CPU_Time_Exam_1}
\begin{tabular}{||c|c|c|c||}
\hline
{\footnotesize Number of levels } &{\footnotesize Number of elements}&{\footnotesize time for ARpack}&
{\footnotesize time for Algorithm \ref{Multi_Correction}}\\
\hline
4    &   253952   &  7.69     &    2.30     \\
\hline
5    &   1015808  &  35.96    &    7.98     \\
\hline
6    &   4063232  &  191.47   &    31.64    \\
\hline
7    &   16252928 &  1258.35  &    127.08   \\
\hline
\end{tabular}
\end{table}

We also check the convergence behavior of multi eigenvalue approximations with Algorithm
\ref{Multi_Correction}. Here the first six eigenvalues
$\lambda=2\pi^2, 5\pi^2, 5\pi^2, 8\pi^2,$ $10\pi^2, 10\pi^2$
are investigated. We adopt the meshes in Figure \ref{Initial_Mesh} as
the initial meshes and the corresponding numerical results are shown
in Figure \ref{numerical_multi_grid_2D_6} which also exhibits the
optimal convergence rate of the multigrid scheme.
\begin{figure}[httb]
\centering
\includegraphics[width=7cm,height=6cm]{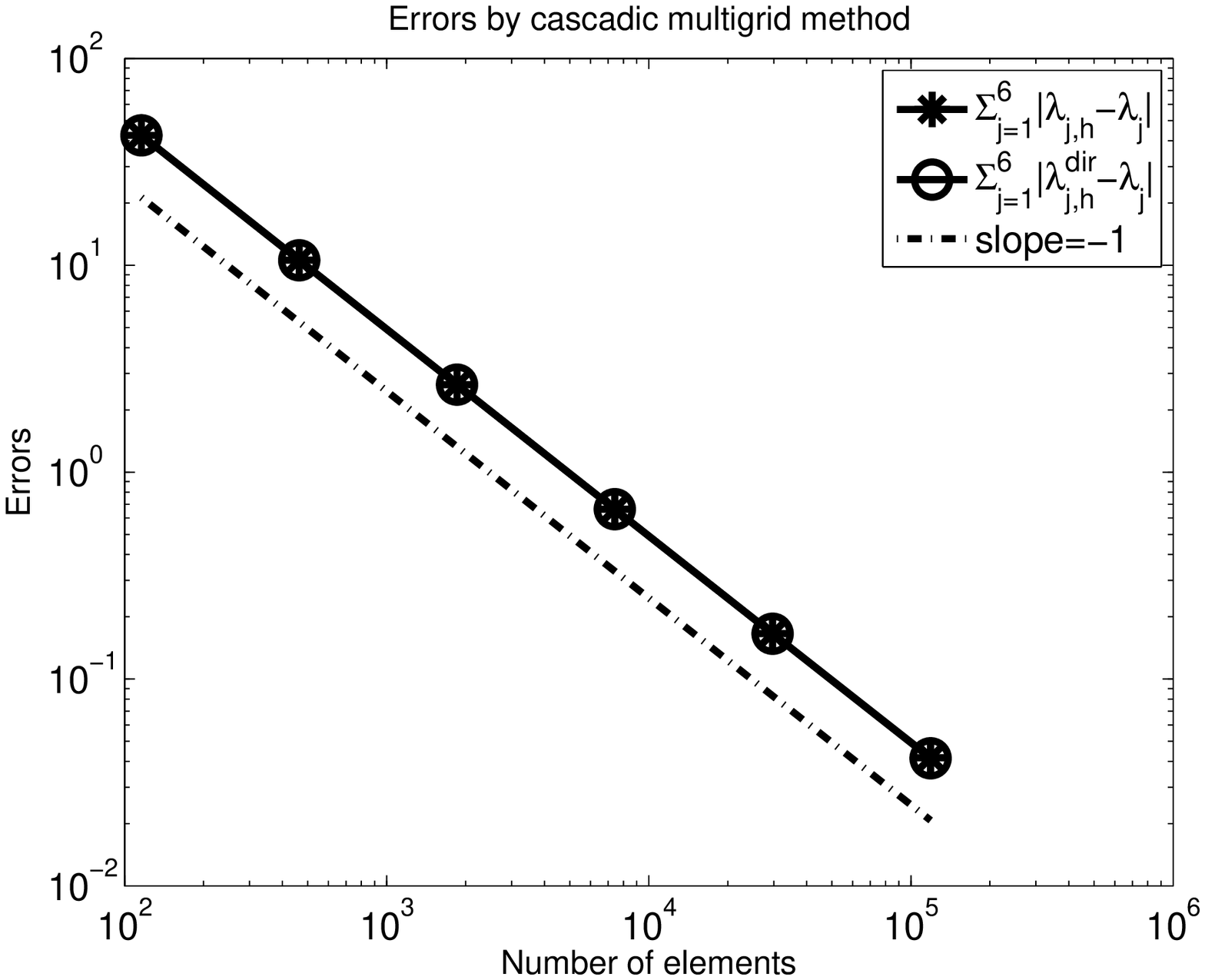}
\includegraphics[width=7cm,height=6cm]{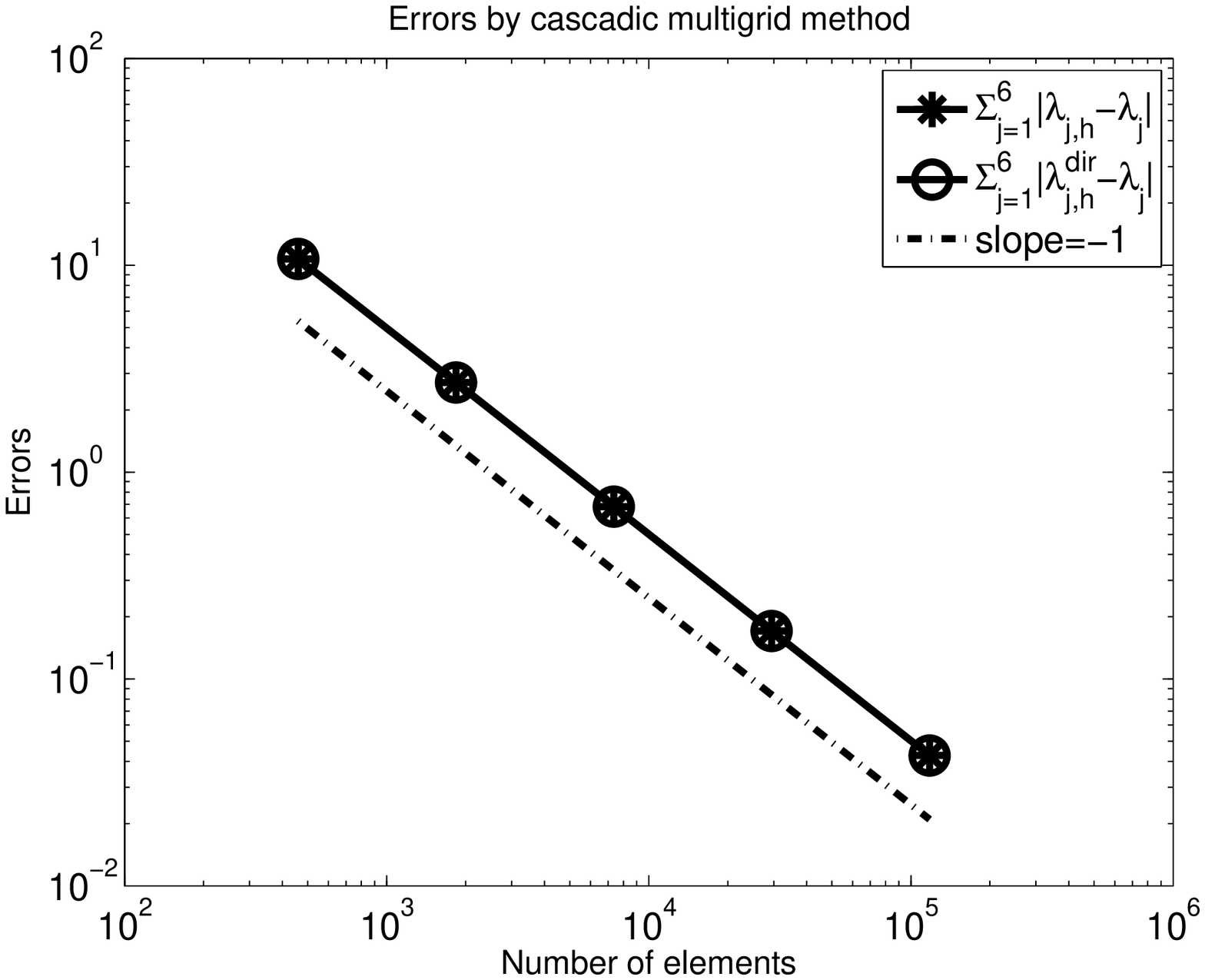}
\caption{\small\texttt The errors of the multigrid
algorithm for the first six eigenvalues on the unit square
(The left subfigure is for the coarse initial mesh in the left of Figure \ref{Initial_Mesh}
and the right one for the fine initial mesh in the right of Figure
\ref{Initial_Mesh}).}\label{numerical_multi_grid_2D_6}
\end{figure}

\subsection{More general eigenvalue problem}
Here we give the numerical results of the multigrid
scheme for solving a more general eigenvalue problem on the unit square
domain $\Omega=(0, 1)\times (0, 1)$:

Find $(\lambda,u)$ such that
\begin{equation}\label{Example_2}
\left\{
\begin{array}{rcl}
-\nabla\cdot\mathcal{A}\nabla u+\phi u&=&\lambda\rho u,\quad{\rm in}\ \Omega,\\
u&=&0,\quad\ \ \ \ {\rm on}\ \partial\Omega,\\
\int_{\Omega}\rho u^2d\Omega&=&1,
\end{array}
\right.
\end{equation}
where
\begin{equation*}
\mathcal{A}=\left (
\begin{array}{cc}
$$1+(x_1-\frac{1}{2})^2$$&$$(x_1-\frac{1}{2})(x_2-\frac{1}{2})$$\\
$$(x_1-\frac{1}{2})(x_2-\frac{1}{2})$$&$$1+(x_2-\frac{1}{2})^2$$
\end{array}
\right),
\end{equation*}
$\phi={\rm e}^{(x_1-\frac{1}{2})(x_2-\frac{1}{2})}$ and
$\rho=1+(x_1-\frac{1}{2})(x_2-\frac{1}{2})$.

We first solve the eigenvalue problem
(\ref{Example_2}) in the linear finite element space on the coarse mesh
$\mathcal{T}_{h_1}$. Then refine the mesh by the regular way to produce
a series of meshes $\mathcal{T}_{h_k}\ (k=2,\cdots,n)$ with $\beta =2$
(connecting the midpoints of each edge) and solve the auxiliary boundary
 value problem (\ref{aux_problem}) in the finer linear finite element space
$V_{h_k}$ defined on $\mathcal{T}_{h_k}$.

In this example, we also use two coarse meshes which are shown in Figure \ref{Initial_Mesh}
as the initial meshes to investigate the convergence behaviors.
%\begin{figure}[htb]
%\centering
%\includegraphics[width=6cm,height=5cm]{Initial_Mesh_Shift_1_12.ps}
%\includegraphics[width=6cm,height=5cm]{Initial_Mesh_Shift_1_16.ps}
%\caption{\small\texttt The coarse and fine initial meshes for Example 2}
%\label{Initial_Mesh_Exam_2}
%\end{figure}
Since the exact solution is not known, we choose an adequately accurate eigenvalue
approximations with the extrapolation method (see, e.g., \cite{LinLin}) as the exact eigenvalue.
Figure \ref{numerical_multi_grid_Exam_2} gives the corresponding
numerical results for the first six eigenvalue approximations and
their corresponding eigenfunction approximations.
Here we also compare the numerical results with the direct algorithm.
Figure \ref{numerical_multi_grid_Exam_2} also exhibits the optimal convergence rate of
Algorithm \ref{Multi_Correction}.
\begin{figure}[htb]
\centering
\includegraphics[width=7cm,height=6cm]{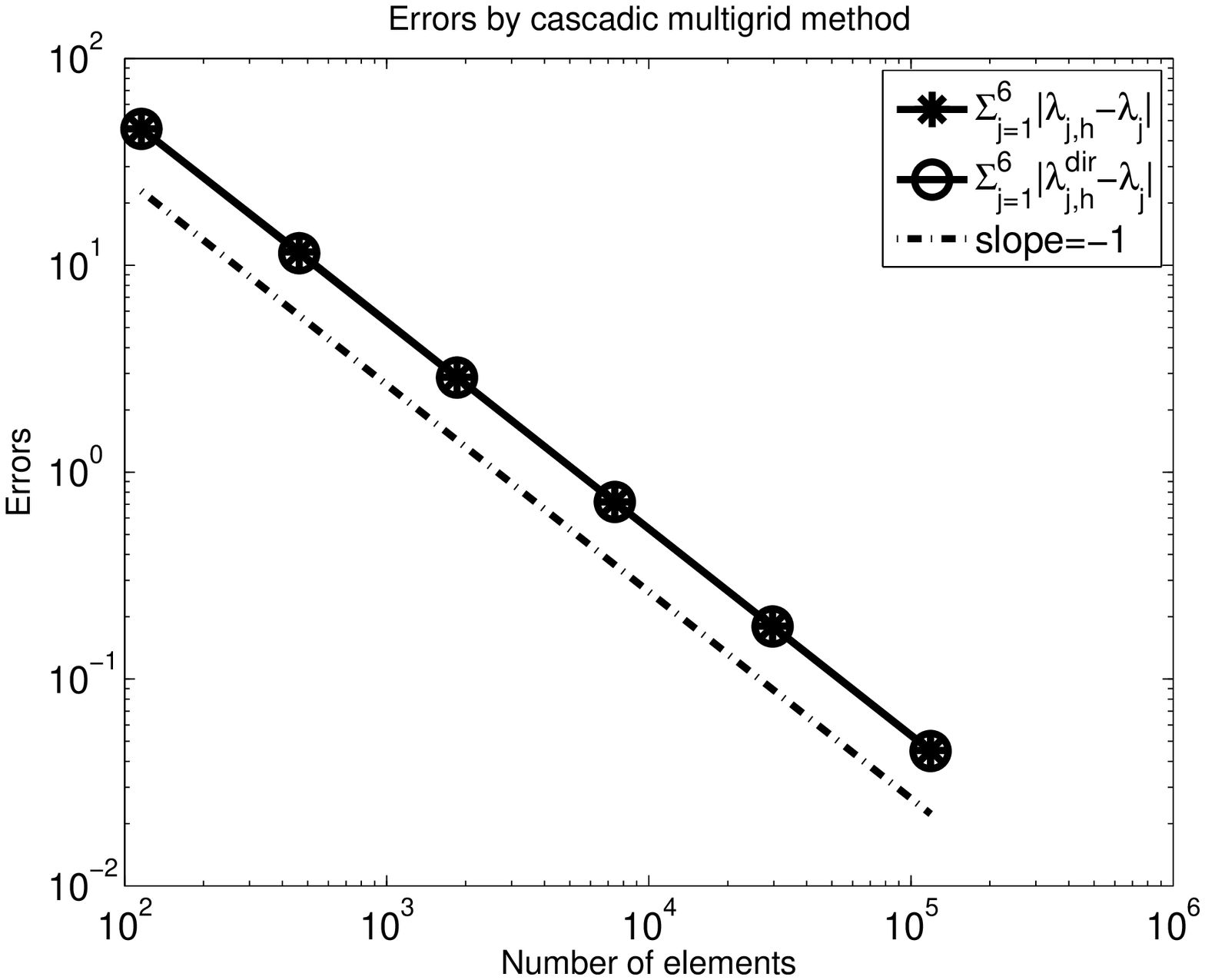}
\includegraphics[width=7cm,height=6cm]{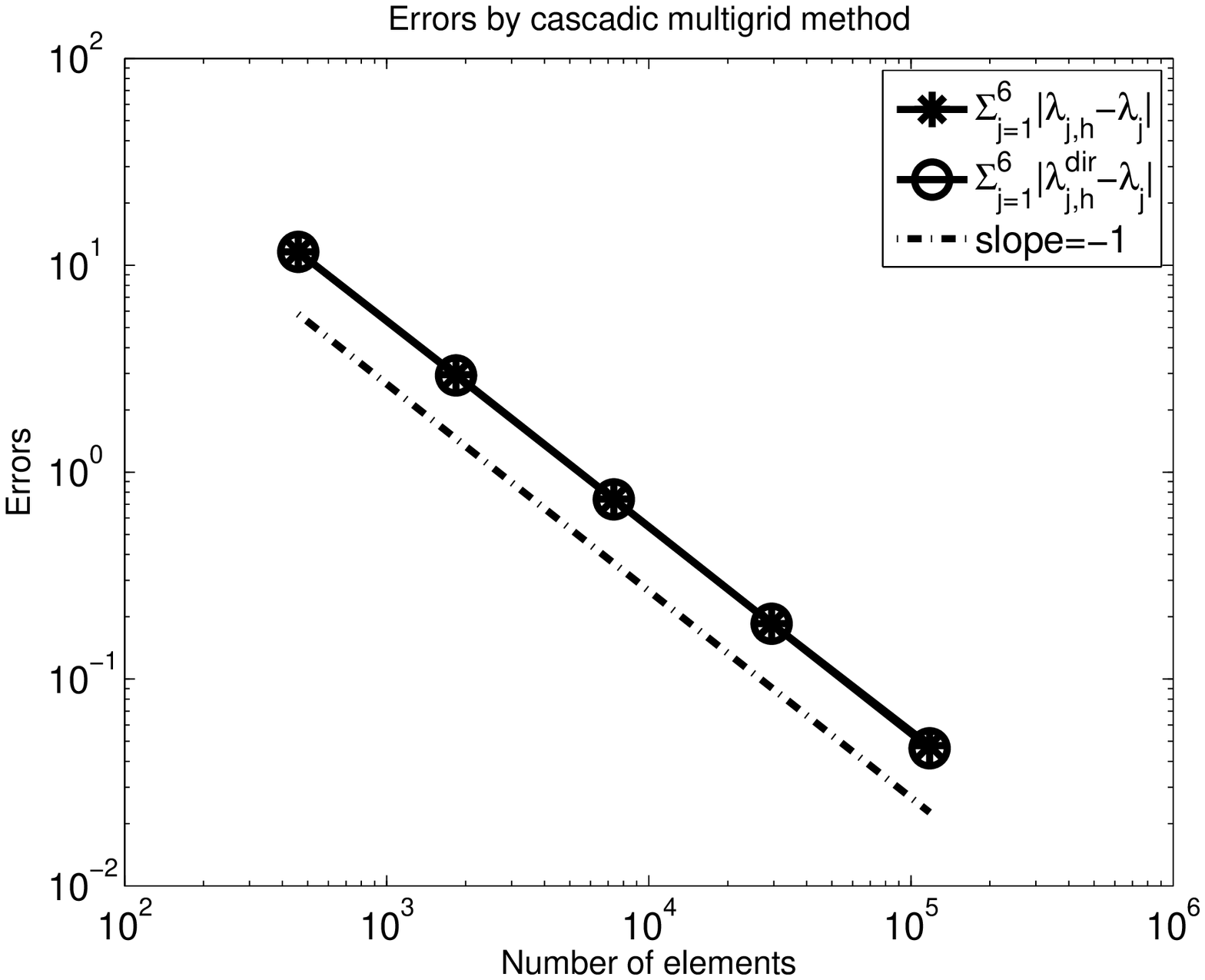}
\caption{\small\texttt The errors of the multigrid
algorithm for the first six eigenvalues and the first eigenfunction,
where $u_h^{\rm dir}$ and $\lambda_h^{\rm dir}$ denote the eigenfunction
and eigenvalue approximation by direct eigenvalue solving (The left subfigure
is for the coarse initial mesh in the left of Figure \ref{Initial_Mesh}
and the right one for the fine initial mesh in the right of
 Figure \ref{Initial_Mesh}).}\label{numerical_multi_grid_Exam_2}
\end{figure}

\section{Concluding remarks}
In this paper, we give a type of multigrid scheme
to solve eigenvalue problems. The idea here is to combine the shifted-inverse
power iteration method with multilevel meshes to transform the solution of
eigenvalue problem to a series of solutions of the corresponding boundary
value problems which can be done by the multigrid method. As stated in the
numerical examples, {\it Eigenvalue Multigrid Scheme} defined in Algorithm
\ref{Multi_Correction} for one eigenvalue can be extended to the corresponding
version for multi eigenvalues (include simple and multiple eigenvalues).
We state the following version of {\it Eigenvalue Multigrid Scheme} for
$m$ eigenvalues.

Similarly, we first give a type of {\it One Correction Step for Multi Eigenvalues}
for the given eigenpairs approximations $\{\lambda_{j,h_k},u_{j,h_k}\}_{j=1}^{m}$.

\begin{algorithm}\label{Correction_Step_Multiple}
One Correction Step for Multi Eigenvalues
\begin{enumerate}
\item Do $j=1,\cdots,m$
\begin{itemize}
\item Find
$\widetilde{u}_{j,h_{k+1}}\in V_{h_{k+1}}$, such that $\forall v_{h_{k+1}}\in V_{h_{k+1}}$
\begin{eqnarray}\label{aux_problem_Multiple}
a(\widetilde{u}_{j,h_{k+1}},v_{h_{k+1}})-\alpha_{j,k+1}b(\widetilde{u}_{j,h_{k+1}},v_{h_{k+1}})
=b(u_{j,h_k},v_{h_{k+1}}).
\end{eqnarray}
\item Do the following orthogonalization and normalization
\begin{eqnarray}\label{Orthogonalization}
\widehat{u}_{j,h_{k+1}}&:=&\widetilde{u}_{j,h_{k+1}}-\sum_{\ell=1}^{j-1}
a(\widetilde{u}_{j,h_{k+1}},u_{\ell,h_{k+1}})u_{\ell,h_{k+1}},\\
u_{j,h_{k+1}}&:=&\frac{\widehat{u}_{j,h_{k+1}}}{\|\widehat{u}_{j,h_{k+1}}\|_a}.
\end{eqnarray}
%\item Do the following
%\begin{eqnarray}
%
%\end{eqnarray}
\end{itemize}
End Do
\item Compute the new eigenvalue approximations
\begin{eqnarray}\label{Eigenvalue_ell}
\lambda_{j,h_{k+1}}=\frac{a(u_{j,h_{k+1}},u_{j,h_{k+1}})}{b(u_{j,h_{k+1}},u_{j,h_{k+1}})},
\ \ \ {\rm for}\ j=1, \ldots, m.
\end{eqnarray}
\end{enumerate}
We summarize the above two steps into
\begin{eqnarray*}
\{\lambda_{j,h_{k+1}},u_{j,h_{k+1}}\}_{j=1}^{m}={\it
Correction}(\{\alpha_{j,k+1}\}_{j=1}^m,\{\lambda_{j,h_k},u_{j,h_k}\}_{j=1}^{m},V_{h_{k+1}}).
\end{eqnarray*}
\end{algorithm}
Similarly to (\ref{Correction_Step_Multi_Steps}), we can also define a modified one correction
 step for multi eigenvalues where we need to solve a small dimensional eigenvalue problem.

\begin{algorithm}\label{Correction_Step_Multiple_Modified}
One Correction Step for Multi Eigenvalues
\begin{enumerate}
\item Do $j=1,\cdots,m$
\begin{itemize}
\item[] Find
% Obtain a new eigenfunction approximation
$\widetilde{u}_{j,h_{k+1}}\in V_{h_{k+1}}$, such that $\forall v_{h_{k+1}}\in V_{h_{k+1}}$ satisfying
%Find $\widetilde{u}_{j,h_{k+1}}\in V_{h_{k+1}}$ such that
\begin{eqnarray}\label{aux_problem_Multiple}
a(\widetilde{u}_{j,h_{k+1}},v_{h_{k+1}})-\alpha_{j,k+1}b(\widetilde{u}_{j,h_{k+1}},v_{h_{k+1}})
=b(u_{j,h_k},v_{h_{k+1}}).
\end{eqnarray}
\end{itemize}
End Do
%Find $\widetilde{u}_{j,h_{k+1}}\in V_{h_{k+1}}$ such that
%\begin{eqnarray}\label{aux_problem_Multiple}
%a(\widetilde{u}_{j,h_{k+1}},v_{h_{k+1}})-\alpha_{j,k+1}b(\widetilde{u}_{j,h_{k+1}},v_{h_{k+1}})
%=b(u_{j,h_k},v_{h_{k+1}}),\ \forall v_{h_{k+1}}\in V_{h_{k+1}}.
%\end{eqnarray}
%Solve this equation to obtain a new eigenfunction approximation
%$\widetilde{u}_{j,h_{k+1}}\in V_{h_{k+1}}$.
\item Build a finite dimensional space
$\widetilde{V}_{h_{k+1}}={\rm span}\{\widetilde{u}_{1,h_{k+1}},
 \cdots, \widetilde{u}_{m,h_{k+1}}\}$
and solve the following eigenvalue problem:

Find $(\lambda_{j,h_{k+1}},u_{j,h_{k+1}})\in
\mathcal{R}\times \widetilde{V}_{h_{k+1}}$, $j=1,2,\cdots,m$, such that
$a(u_{j,h_{k+1}},u_{j,h_{k+1}})=1$ and
\begin{eqnarray}\label{Initial_Eigen_Problem_Multiple}
a(u_{j,h_{k+1}},v_{h_{k+1}})&=&\lambda_{j,h_{k+1}}b(u_{j,h_{k+1}},v_{h_{k+1}}),
\ \ \ \ \forall v_{h_{k+1}}\in \widetilde{V}_{h_{k+1}}.
\end{eqnarray}
\end{enumerate}
We summarize the above two steps into
\begin{eqnarray*}
\{\lambda_{j,h_{k+1}},u_{j,h_{k+1}}\}_{j=1}^{m}={\it
Correction}(\{\alpha_{j,k+1}\}_{j=1}^m,\{\lambda_{j,h_k},u_{j,h_k}\}_{j=1}^{m},V_{h_{k+1}}).
\end{eqnarray*}
\end{algorithm}
Based on Algorithm \ref{Correction_Step_Multiple} or \ref{Correction_Step_Multiple_Modified},
we can give the corresponding
multigrid method.
\begin{algorithm}\label{Multi_Correction_Multiple}
Eigenvalue Multigrid Scheme for Multi Eigenvalues
\begin{enumerate}
\item Construct a series of nested finite element
spaces $V_{h_1}, V_{h_2},\cdots,V_{h_n}$ such that
(\ref{FEM_Space_Series}) and (\ref{Error_k_k_1}) hold.
\item Solve the following eigenvalue problem:

Find $(\lambda_{h_1},u_{h_1})\in \mathcal{R}\times V_{h_1}$ such that
$a(u_{h_1},u_{h_1})=1$ and
\begin{eqnarray}\label{Initial_Eigen_Problem_Multiple}
a(u_{h_1},v_{h_1})&=&\lambda_{h_1}b(u_{h_1},v_{h_1}),\ \ \ \ \forall v_{h_1}\in V_{h_1}.
\end{eqnarray}
Choose $m$ eigenpairs $\{\lambda_{j,h_1},u_{j,h_1}\}_{j=1}^{m}$ which approximate
our desired eigenvalues and their eigenspaces.

\item  Do $k=1,\cdots,n-1$
%\begin{itemize}
%\item[]

Obtain new eigenpair approximations
$\{\lambda_{j,h_{k+1}},u_{j,h_{k+1}}\}_{j=1}^{m}\in \mathcal{R}\times V_{h_{k+1}}$
by a correction step
\begin{eqnarray*}
\{\lambda_{j,h_{k+1}},u_{j,h_{k+1}}\}_{j=1}^{m}={\it
Correction}(\{\alpha_{j,k+1}\}_{j=1}^m,\{\lambda_{j,h_k},u_{j,h_k}\}_{j=1}^{m},V_{h_{k+1}}).
\end{eqnarray*}
%\end{itemize}
End Do
\end{enumerate}
Finally, we obtain $m$ eigenpair approximations
$\{\lambda_{j,h_n},u_{j,h_n}\}_{j=1}^{m}\in \mathcal{R}\times V_{h_n}$.
\end{algorithm}
We can also define
\begin{eqnarray}\label{alpha_defintion_j_1_1}
\alpha_{j,k+1}=\max\left\{0,\frac{2(1+C_4/C_5)
\beta\lambda_{j,h_k}-\lambda_{j+1,h_k}}{2(1+C_4/C_5)\beta-1}\right\},
\ \ \ j=1, \ldots, m-1
\end{eqnarray}
and
\begin{eqnarray}\label{alpha_defintion_j_1_2}
\alpha_{m,k+1}=\max\left\{0,\frac{2(1+C_4/C_5)
\beta\lambda_{m,h_k}-\lambda_{m+1,h_1}}{2(1+C_4/C_5)\beta-1}\right\}.
\end{eqnarray}
Based on the above definitions of $\alpha_{j,k+1}$, we can also give the error analysis for
this version of eigenvalue multigrid method in the similar way used in Sections 3 and 4.
If we use the mulgirid method (the multigrid method for indefinite problems from
\cite{Shaidurov,Xu_Two_Grid}) to solve the boundary value problems included in
Algorithm \ref{Multi_Correction_Multiple}, the computational work involved in the multi
eigenvalues version is $\mathcal{O}(m^2N_n)$.
 Furthermore, the parallel computation can be used to solve (\ref{aux_problem_Multiple})
for  different $j$.
The analysis of the scheme for multi eigenvalues will be given in our future work.

We can replace the multigrid method by other types of efficient iteration schemes
such as algebraic multigrid method, the type of preconditioned schemes based on
the subspace decomposition and subspace corrections
(see, e.g., \cite{BrennerScott, Xu}), and the domain decomposition method
(see, e.g., \cite{ToselliWidlund}). Furthermore, the framework here
can also be coupled with the parallel method and the adaptive refinement
technique. The ideas should be extended to other types of linear eigenvalue
problems. These will be investigated in our future work.


\begin{thebibliography}{17}

\bibitem{Babuska2}
I. Babu\v{s}ka and J. Osborn, {\em Finite element-Galerkin
approximation of the eigenvalues and eigenvectors of selfadjoint
problems}, Math. Comp. 52 (1989), 275--297.

\bibitem{BabuskaOsborn}
I. Babu\v{s}ka and J. Osborn, {\em Eigenvalue Problems}, In Handbook of
Numerical Analysis, Vol. II, (Eds. P. G. Lions and P. G. Ciarlet),
Finite Element Methods (Part 1), North-Holland, Amsterdam, 641--787,
1991.

%\bibitem{BankDupont}
%R. E. Bank and T. Dupont, {\em An optimal order process for solving finite element equations},
%Math. Comp., 36 (1981), 35-51.
%
%\bibitem{Bramble}
%J. H. Bramble, {\em Multigrid Methods}, Pitman Research Notes in Mathematics, V. 294,
%John Wiley and Sons, 1993.
%
%
%\bibitem{BramblePasciak}
%J. H. Bramble and J. E. Pasciak, {\em New convergence estimates for multigrid algorithms},
%Math. Comp. 49 (1987), 311-329.
%
%\bibitem{BrambleZhang}
%J. H. Bramble and X. Zhang, {\em The analysis of Multigrid Methods}, Handbook of Numerical Analysis,
%Vol. VII, P. G. Ciarlet and J. L. Lions, eds., Elsevier Science, 173-415, 2000.

\bibitem{BrandtMcCormickRuge}
A. Brandt, S. McCormick and J. Ruge, {\em Multigrid methods for differential eigenproblems},
SIAM J. Sci. Stat. Comput., 4(2) (1983), 244--260.

\bibitem{BrennerScott}
S. Brenner and L. Scott, {\em The Mathematical Theory of Finite Element
Methods}, New York: Springer-Verlag, 1994.

%\bibitem{ChanSharapov}
%T. F. Chan and I. Sharapov, {\em Subspace correction multi-level
%methods for elliptic eigenvalue problems},
%Numer. Linear Algebra Appl., 9  (2002), 1--20.

\bibitem{Chatelin}
F. Chatelin, {\em Spectral Approximation of Linear Operators}, Academic
Press Inc, New York, 1983.

\bibitem{Ciarlet}
P. G. Ciarlet, {\em The finite Element Method for Elliptic Problem},
North-Holland Amsterdam, 1978.

%\bibitem{Conway}
%J. Conway, {\em A Course in Functional Analysis}, Springer-Verlag, New-York, 1990.

%\bibitem{DaiZhou}
%X. Dai and A. Zhou, Three-scale finite element discretizations for quantum eigenvalue problems,
%SIAM J. Numer. Anal., 46 (2008), 295--324.

\bibitem{Hackbusch}
W. Hackbusch, {\em On the computation of approximate eigenvalues and eigenfunctions of
 elliptic operators by means of
a multi-grid method}, Siam J. Numer. Anal., 16(2) (1979), 201--215.

\bibitem{Hackbusch_Book}
W. Hackbusch, {\em Multi-grid Methods and Applications}, Springer-Verlag, Berlin, 1985.

\bibitem{HuCheng}
X. Hu and X. Cheng,
{\em Acceleration of a two-grid method for eigenvalue problems}, Math. Comp.,
80(275) (2011), 1287--1301.

%\bibitem{McCormick}
%S. F. McCormick, ed., {\em Multigrid Methods}. SIAM Frontiers
% in Applied Matmematics 3. Society for Industrial and
%Applied Mathematics, Philadelphia, 1987.

\bibitem{LehoucqMaschhoffSorensenYang}
R. Lehoucq, K. Maschhoff, D. Sorensen and C. Yang, ARPACK Software Package,
http://www.caam.rice.edu/software/ARPACK/, 1996.

\bibitem{LehoucqSorensenYang}
R. Lehoucq, D. C. Sorensen, and C. Yang, ARPACK Users' Guide, SIAM, Philadelphia, 1998.



\bibitem{LinLin}
Q. Lin and J. Lin, {\em Finite Element Methods: Accuracy and Inprovement}, China Sci.
Tech. Press, 2005. %\textcolor{blue}{(in Chinese)}

%\bibitem{LinXie}
%Q. Lin and H. Xie, {\em A multi-level correction scheme for eigenvalue problems},
%http://arxiv.org/abs/1107.0223, to apper, 2010.

\bibitem{LinXie}
Q. Lin and H. Xie, {\em A multi-level correction scheme for eigenvalue problems},
Math. Comp., doi: S 0025-5718(2014)02825-1, March 10, 2014.


%
%\bibitem{LinXieXu}
%Q. Lin, H. Xie, and J. Xu, {\em Lower bounds of the discretization for piecewise polynomials},
%http://arxiv.org/abs/1106.4395, 2011.


\bibitem{LinYan}
Q. Lin and N. Yan, {\em The Construction and Analysis of High Efficiency
Finite Element Methods}, Hebei University Publishers, 1995. (in Chinese)


%\bibitem{ScottZhang}
%L. R. Scott and S. Zhang, {\em Higher dimensional non-nested multigrid methods}, Math. Comp.,
%58 (1992), 457--466.

\bibitem{Shaidurov}
V. Shaidurov, {\em Multigrid methods for finite element}, Kluwer
Academic Publ., Netherlands, 1995.

\bibitem{ToselliWidlund}
A. Toselli and O. Widlund, {\em Domain Decomposition Methods: Algorithm and Theory},
Springer-Verlag, Berlin Heidelberg, 2005.


\bibitem{Xie_IMA}
H. Xie, {\em A type of multilevel method for the Steklov eigenvalue problem},
IMA J. Numer. Anal., 34(2) (2014), 592--608.


\bibitem{Xie_JCP}
H. Xie, {\em A multigrid method for eigenvalue problem}, J. Comput. Phys., 274 (2014),
550--561.


\bibitem{Xu}
J. Xu, {\em Iterative methods by space decomposition and subspace
correction}, SIAM Review, 34(4) (1992), 581--613.

\bibitem{Xu_Two_Grid}
J. Xu, {\em A new class of iterative methods for nonselfadjoint or indefinite problems},
SIAM J. Numer. Anal., 29 (1992), 303--319.

%\bibitem{Xu_Nonlinear}
%J. Xu, {\em A novel two-grid method for semilinear elliptic equations},
%SIAM J. Sci. Comput., 15 (1994), 231--237.

\bibitem{XuZhou}
J. Xu and A. Zhou, {\em A two-grid discretization scheme for eigenvalue
problems}, Math. Comput., 70 (233) (2001), 17--25.

%\bibitem{XuZhou_Eigen}
%J. Xu and A. Zhou, {\em Local and parallel finite element algorithm for
%eigenvalue problems}, Acta Math. Appl. Sin. Engl. Ser.,  18(2) (2002),
%185-200.

\bibitem{YangBi}
Y. Yang and H. Bi,  {\em Two-grid finite element discretization schemes based on
shifted-inverse power method for elliptic eigenvalue problems},
SIAM J. Numer. Anal., 49(4) (2011), 1602--1624.

%\bibitem{Yserentant}
%H. Yserentant, {\em On the multilevel splitting of finite element spaces},
% Numer. Math., 49 (1986), 379--412.
%
%\bibitem{Zhou}
%A. Zhou, {\em multi-level adaptive corrections in finite dimensional
%approximations}, J. Comp. Math., 28(1) (2010), 45--54.
\end{thebibliography}
\end{document}